\documentclass{amsart}
\usepackage{amssymb,amsmath,amsthm,amscd}
\usepackage{graphicx}
\usepackage{color}
\usepackage[usenames,dvipsnames,svgnames,table]{xcolor}
\usepackage{tikz}
\usepackage{comment}
\usepackage{epstopdf}
\usepackage{float}
\usepackage{tikz-cd}
\newtheorem{lemma}{Lemma}[section]
\newtheorem{theorem}[lemma]{Theorem}
\newtheorem{corollary}[lemma]{Corollary}
\theoremstyle{definition}
\newtheorem*{definition}{Definition}
\theoremstyle{remark}
\newtheorem*{remark}{Remark}

\renewcommand{\epsilon}{\varepsilon}

\title[Non-landing rays]{Non-landing parameter rays of the Multicorns}

\author[H.~Inou]{Hiroyuki Inou}
\address{Department of Mathematics, Kyoto University, Kyoto 606-8502, Japan}
\email{inou@math.kyoto-u.ac.jp}
\thanks{}

\author[S.~Mukherjee]{Sabyasachi Mukherjee}
\address{Jacobs University Bremen, Campus Ring 1, Bremen 28759, Germany}
\email{s.mukherjee@jacobs-university.de}


\date{\today}

\begin{document}

\begin{abstract}
It is well known that every rational parameter ray of the Mandelbrot set lands at a single parameter. We study the rational parameter rays of the multicorn $\mathcal{M}_d^*$, the connectedness locus of unicritical antiholomorphic polynomials of degree $d$, and give a complete description of their accumulation properties. One of the principal results is that the parameter rays accumulating on the boundaries of odd period (except period $1$) hyperbolic components of the multicorns do not land, but accumulate on arcs of positive length consisting of parabolic parameters.

We also show the existence of undecorated real-analytic arcs on the boundaries of the multicorns, which implies that the centers of hyperbolic components do not accumulate on the entire boundary of $\mathcal{M}_d^*$, and the Misiurewicz parameters are not dense on the boundary of $\mathcal{M}_d^*$.
\end{abstract}

\maketitle

\tableofcontents

\section{Introduction}
We consider unicritical antiholomorphic polynomials $f_c(z) = \overline{z}^d + c$ for any degree $d \geq 2$, and $c \in \mathbb{C}$. In analogy to the holomorphic case, the set of all points which remain bounded under all iterations of $f_c$ is called the \emph{filled-in Julia set} $K(f_c)$. The boundary of the filled-in Julia set is defined to be the \emph{Julia set} $J(f_c)$, and the complement of the Julia set is defined to be its \emph{Fatou set} $F(f_c)$. This leads, as in the holomorphic case, to the notion of \emph{connectedness locus} of degree $d$ unicritical antiholomorphic polynomials:

\begin{definition}
The \emph{multicorn} of degree $d$ is defined as $\mathcal{M}^{\ast}_d = \{ c \in \mathbb{C} : K(f_c)$ is connected$\}.$
\end{definition} 

The dynamics of quadratic antiholomorphic polynomials, and its connectedness locus, $\mathcal{M}^*_2$ (also known as the tricorn), was first studied in \cite{CHRS}, and their numerical experiments showed differences between the Mandelbrot set and the tricorn in that there are bifurcations from the period $1$ hyperbolic component to period $2$ hyperbolic components along arcs in the latter. However, it was Milnor who first observed the importance of the multicorns; he found little tricorn and multicorn-like sets as prototypical objects in the parameter space of real cubic polynomials \cite{M3} and in the real slices of rational maps with two critical points \cite{M4}. Nakane \cite{Na1} proved that the tricorn is connected, in analogy to Douady and Hubbard's classical proof on the Mandelbrot set. This generalizes naturally to multicorns of any degree. Later, Nakane and Schleicher, in \cite{NS}, studied the structure of hyperbolic components  of $\mathcal{M}_d^*$ via the multiplier map (even period case), and the critical value map (odd period case). These maps are branched coverings over the unit disk of degree $d-1$ and $d+1$ respectively, branched only over the origin. Hubbard and Schleicher \cite{HS} proved that the multicorns are not pathwise connected, confirming a conjecture of Milnor. Recently, in an attempt to explore the topological aspects of the parameter spaces of unicritical antiholomorphic polynomials, the combinatorics of external dynamical rays of such maps were studied in \cite{Sa} in terms of orbit portraits, and this was used in \cite{MNS} where the bifurcation phenomena, boundaries of odd period hyperbolic components, and the combinatorics of parameter rays were described. In \cite{I1,IM}, we study the `universality' property of the multicorns, and give a precise explanation for the existence of `baby multicorns' in the multicorns, and in the parameter space of real cubic polynomials. The main results of \cite{IM} show that the straightening map from a `baby multicorn', either in multicorns of even degree or in the real cubic locus, to the original multicorn is discontinuous. A computer-assisted proof of this phenomenon for a particular period $3$ baby tricorn has been given in \cite{I1}. These are the first known examples where the straightening map fails to be continuous on a real two-dimensional slice of a holomorphic family of holomorphic polynomials. The proof of discontinuity of the straightening map in \cite{IM} is carried out by showing that all non-real \emph{umbilical cords} of the multicorns wiggle, which generalizes a theorem of \cite{HS}, and settles a conjecture made by various people including Hubbard, Inou, Milnor and Schleicher. 

The combinatorics and topology of the multicorns differ in many ways from those of their holomorphic counterparts, the multibrot sets, which are the connectedness loci of degree $d$ unicritical polynomials. At the level of combinatorics, this is manifested in the structure of orbit portraits \cite[Theorem 2.6, Theorem 3.1]{Sa}. The topological features of the multicorns have quite a few properties in common with the connectedness locus of real cubic polynomials, e.g.\ discontinuity of landing points of dynamical rays, bifurcation along arcs, existence of real-analytic curves containing q.c.-conjugate parabolic parameters, lack of local connectedness of the connectedness loci etc. \cite{lavaurs_systemes_1989}, \cite{KN}, \cite[Corollary 3.7]{HS}, \cite[Theorem 3.2, Theorem 6.2]{MNS}. These are in stark contrast with the multibrot sets. 

\medskip

\begin{figure}[ht!]
\includegraphics[scale=0.34]{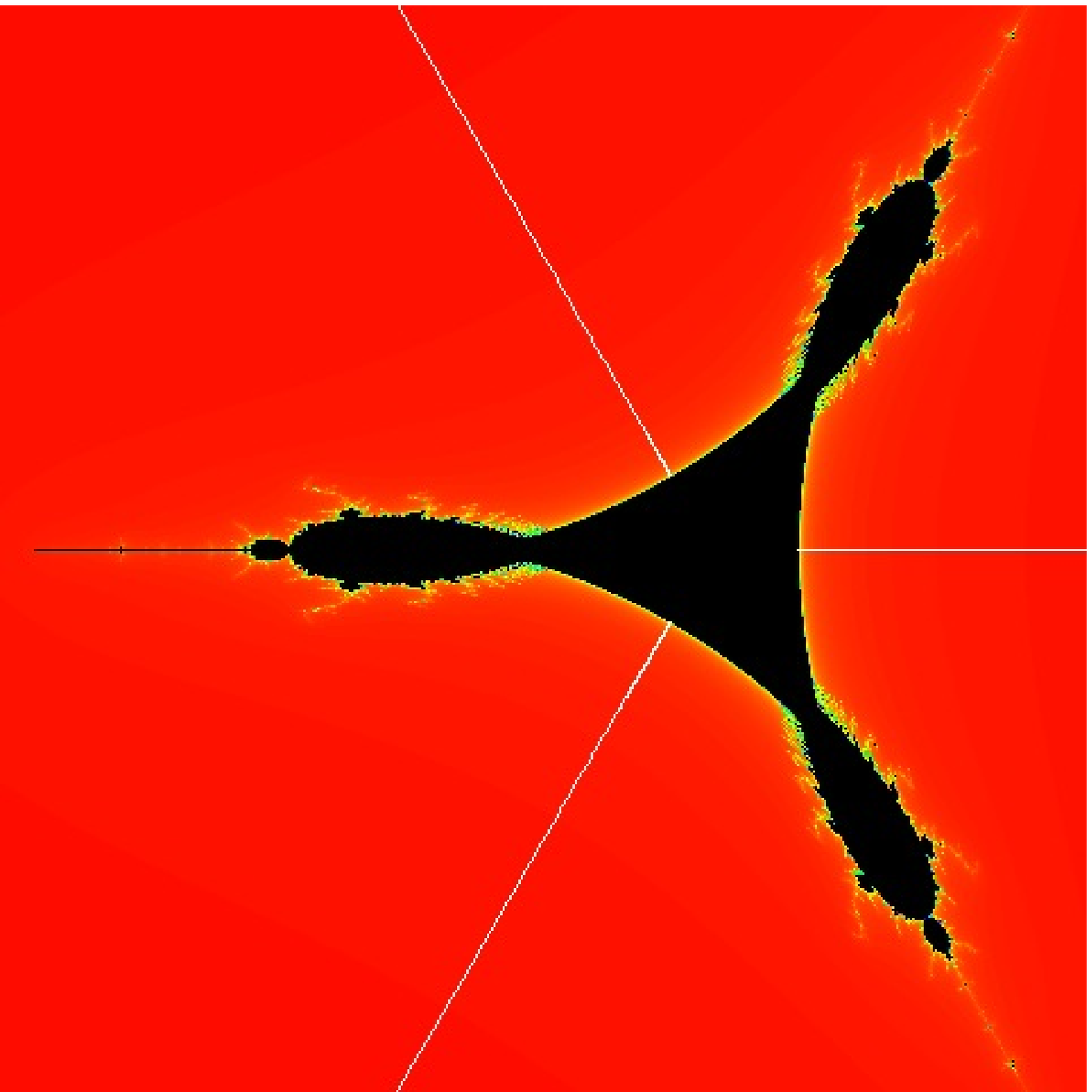} \hspace{4mm} \includegraphics[scale=0.34]{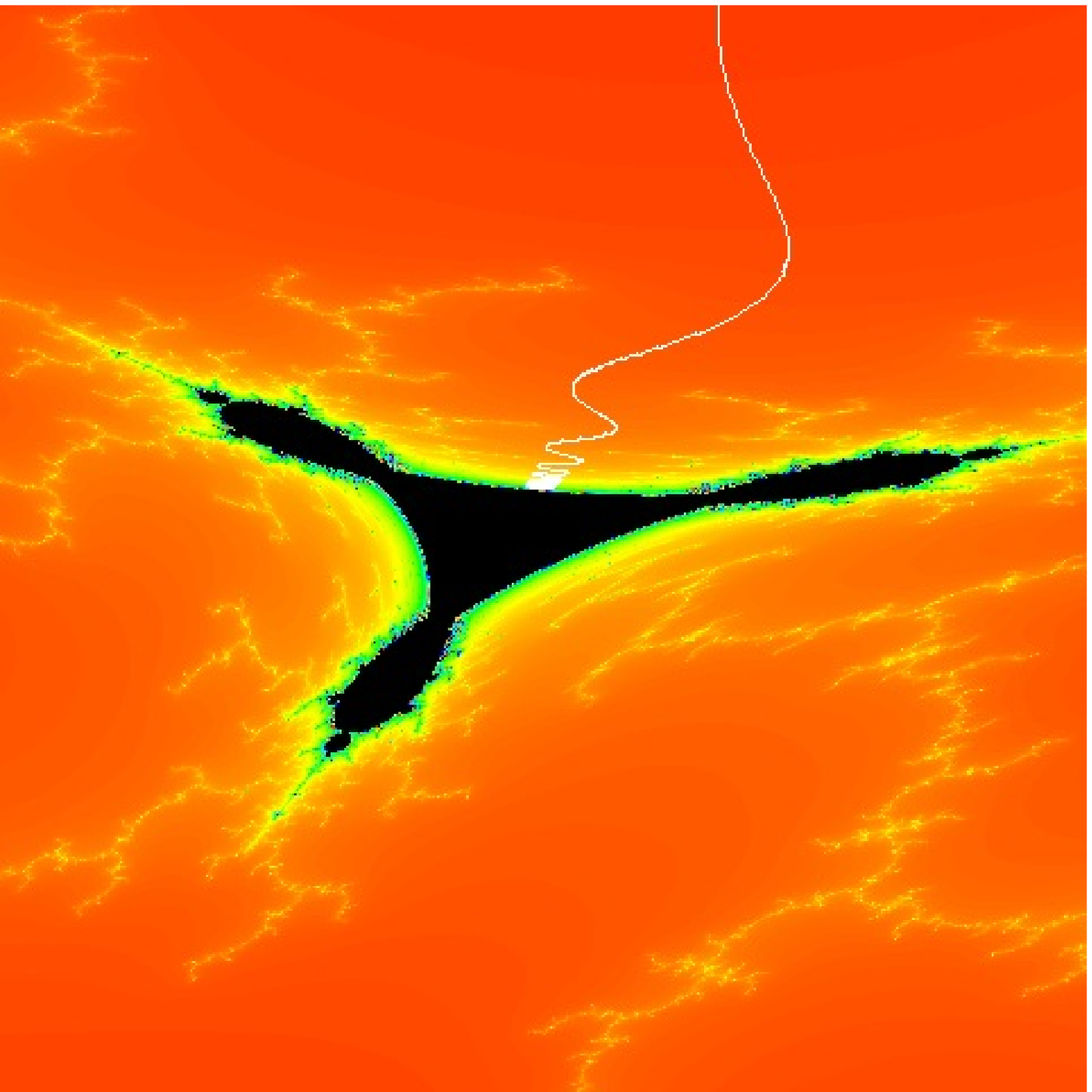}
\label{landwiggle}
\caption{Left: Landing of the parameter rays at fixed angles on parabolic arcs containing undecorated sub-arcs. Right: Non-trivial accumulation of a parameter ray at an odd-periodic angle.}
\end{figure}

One of the main purposes of this paper is to give a complete description of the landing/accumulation properties of the rational parameter rays of $\mathcal{M}^{\ast}_d$ (see Theorem \ref{complete picture of rays}). The landing of rational parameter rays of the Mandelbrot set is related to the fact that the parabolic parameters with given combinatorics are isolated in the Mandelbrot set \cite[Theorem C.7]{GM1}, \cite[Proposition 3.1]{S1a}. This remains true for even period parabolics of the multicorns. Hence, if the accumulation set of a parameter ray $\mathcal{R}_t^d$ of $\mathcal{M}^{\ast}_d$ contains a parameter $c$ having an even-periodic parabolic cycle, then it lands at a single point. This statement was proved in \cite[Lemma 7.2]{MNS}. But the odd periodic parabolic parameters of the multicorns are far from being isolated, there are real-analytic arcs of combinatorially equivalent parabolic parameters of any given odd period. This, at least heuristically, tells that there is no good reason for a parameter ray to land at a single point of a parabolic arc, unless it does so for some symmetry reasons. The following theorem confirms this heuristics (compare Figure \ref{landwiggle}).

\begin{theorem}[Non-Landing Parameter Rays]\label{most rays wiggle}
The accumulation set of every parameter ray accumulating on the boundary of a hyperbolic component of odd period (except period one) of $\mathcal{M}_d^{\ast}$ contains an arc of positive length.
\end{theorem}

The wiggling behavior is stated precisely, and proved in Section \ref{wiggle}. Its proof is based on the analysis of certain geometric properties of the repelling Fatou coordinates, and transferring them to the parameter plane by a perturbation argument. Section \ref{combinatorics_tells_wiggling} gives a complete description of which rational parameter rays of the multicorns land, and which ones have this wiggling property, in terms of the combinatorics of the angle.

It is worth noting that non-trivial accumulation of some stretching rays in the parameter space of real cubic polynomials was proved by Nakane and Komori in \cite{KN} by different methods. It has been empirically observed that there are infinitely many small tricorn-like sets in the parameter space of real cubic polynomials. Our techniques can be naturally generalized to these small tricorn-like sets (of course, these need to be defined rigorously) yielding the non-trivial accumulation of the stretching rays that approach the parabolic arcs on the boundaries of these small tricorn-like sets. More generally, one expects the existence of tricorn-like sets in any family of polynomials or rational maps with (at least) two critical orbits such that a pair of critical orbits are symmetric with respect to an antiholomorphic involution. In a recent unpublished manuscript \cite{BBM}, Bonifant, Buff, and Milnor studied the parameter space of antipode preserving cubic rational maps, and one can observe many tricorn-like sets appearing there.

In the last section, we will study another topological property of the multicorns. In \cite{HS}, it was asked whether the parabolic arcs of the multicorns can contain undecorated sub-arcs. More precisely, we say that a parabolic arc $\mathcal{C}$ on the boundary of a hyperbolic component $H$ of odd period of $\mathcal{M}_d^*$ contains an undecorated sub-arc if there exists a parameter $c$ lying on $\mathcal{C}\subset \partial H$, and an open neighborhood $U$ of $c$ such that $\overline{U}\setminus \overline{H}$ is contained in the complement of $\mathcal{M}_d^*$; i.e. $(\overline{U}\setminus \overline{H})\cap \mathcal{M}_d^*=\emptyset$. In section \ref{undecorated}, we answer this question affirmatively by showing that:
 
\begin{theorem}[Undecorated Arcs on The Boundary]\label{period 1 multicorns}
For $d \geq 2$, every period $1$ parabolic arc of $\mathcal{M}_d^*$ contains an undecorated sub-arc.
\end{theorem}

Once again, the key idea is to transfer a geometric property of the repelling Ecalle cylinder to the parameter plane using perturbation techniques. There are some interesting consequences of the previous theorem. One of them is that the centers of hyperbolic components as well as the Misiurewicz parameters (with strictly pre-periodic critical points) are not dense on the boundary of $\mathcal{M}^{\ast}_d$ (Corollary \ref{centers not dense}). This is another item in the list of the topological differences between the multicorns and the multibrot sets. The fact that the centers of hyperbolic components (or the Misiurewicz parameters) are dense on the boundaries of the multibrot sets follows by an easy application of Montel's theorem.

We would like to thank Adam Epstein for many helpful discussions. Special thanks go to Dierk Schleicher for his useful suggestions to improve the original manuscript, and for allowing us to reproduce one of the figures from \cite{HS}. The second author gratefully acknowledges the support of Deutsche Forschungsgemeinschaft DFG during this work.

\section{Parabolic Points in Antiholomorphic Dynamics}\label{Secbackground}
In this section, we briefly recall some known results on antiholomorphic dynamics and their parameter spaces, which we will have need for in the rest of the paper. The next result was proved by Nakane (see \cite{Na1}).
\begin{theorem}[Real-Analytic Uniformization]\label{RealAnalUniformization}
The map $\Phi : \mathbb{C} \setminus \mathcal{M}^{\ast}_d \rightarrow \mathbb{C} \setminus \overline{\mathbb{D}}$, defined by $c \mapsto \phi_c(c)$ (where $\phi_c$ is the B\"{o}ttcher coordinate near $\infty$ for $f_c$, normalized so that it is tangent to the identity at $\infty$) is a real-analytic diffeomorphism. In particular, the multicorns are connected.
\end{theorem}

The previous theorem also allows us to define parameter rays of the multicorns. 
\begin{definition}[Parameter Ray]
The parameter ray at angle $\theta$ of the multicorn $\mathcal{M}^{\ast}_d$, denoted by $\mathcal{R}_{\theta}^d$, is defined as $\{ \Phi^{-1}(r e^{2 \pi i \theta}) : r > 1 \}$, where $\Phi$ is the real-analytic diffeomorphism from the exterior of $\mathcal{M}_d^*$ to the exterior of the closed unit disc in the complex plane constructed in Theorem \ref{RealAnalUniformization}.
\end{definition}

\begin{remark}
Some comments should be made on the definition of the parameter rays. Observe that unlike the multibrot sets, the parameter rays of the multicorns are not defined in terms of the Riemann map of the exterior. In fact, the Riemann map of the exterior of $\mathcal{M}_d^*$ has no obvious dynamical meaning; we have defined the parameter rays via a dynamically defined diffeomorphism of the exterior of $\mathcal{M}_d^*$. Note that since the second iterate of $f_c(z)=\overline{z}^d+c$ is the holomorphic polynomial $f_c^{\circ 2}(z)= (z^d+\overline{c})^d+c$, the family $\{f_c\}_{c\in\mathbb{C}}$ sits as the real two-dimensional slice $a=\overline{b}$ in the holomorphic family of polynomials $\{(z^d+a)^d+b\}_{a,b\in\mathbb{C}}$. It is now easy to check that our definition of parameter rays of the multicorns agrees with the notion of stretching rays (compare \cite{Lei,KN}) in the family of polynomials $\{(z^d+a)^d+b\}_{a,b\in\mathbb{C}}$.
\end{remark}

One of the main features of the antiholomorphic parameter spaces is the existence of abundant parabolics. In particular, the boundaries of odd period hyperbolic components of the multicorns consist only of parabolic parameters.
\begin{lemma}[Indifferent Dynamics of Odd Period]\label{LemOddIndiffDyn}  
The boundary of a hyperbolic component of odd period $k$ consists 
entirely of parameters having a parabolic orbit of exact period $k$. In appropriate 
local conformal coordinates, the $2k$-th iterate of such a map has the form 
$z\mapsto z+z^{q+1}+\ldots$ with $q\in\{1,2\}$. 
\end{lemma}

\begin{proof} 
See \cite[Lemma 2.8]{MNS}.
\end{proof}

This leads to the following classification of odd periodic parabolic points.
\begin{definition}[Parabolic Cusps]\label{DefCusp}
A parameter $c$ will be called a {\em cusp point} if it has a parabolic 
periodic point of odd period such that $q=2$ in the previous lemma. Otherwise, it is called a \emph{simple} parabolic parameter.
\end{definition}

In holomorphic dynamics, the local dynamics in attracting petals of parabolic periodic points is well-understood: there is a local coordinate $\zeta$ which conjugates the first-return dynamics to the form $\zeta\mapsto\zeta+1$ in a right half place (see Milnor~\cite[Section~10]{M1new}. Such a coordinate $\zeta$ is called a \emph{Fatou coordinate}. Thus the quotient of the petal by the dynamics is isomorphic to a bi-infinite cylinder, called the \emph{Ecalle cylinder}. Note that Fatou coordinates are uniquely determined up to addition by a complex constant. 

In antiholomorphic dynamics, the situation is at the same time restricted and richer. Indifferent dynamics of odd period is always parabolic because for an indifferent periodic point of odd period $k$, the $2k$-th iterate is holomorphic with positive real multiplier, hence parabolic as described above. On the other hand, additional structure is given by the antiholomorphic intermediate iterate. 

\begin{lemma}[Fatou Coordinates]\label{normalization of fatou}
Suppose $z_0$ is a parabolic periodic point of odd period $k$ of $f_c$ with only one petal (i.e.\ $c$ is not a cusp), and $U$ is a periodic Fatou component with $z_0 \in \partial U$. Then there is an open subset $V \subset U$ with $z_0 \in \partial V$, and $f_c^{\circ k} (V) \subset V$ so that for every $z \in U$, there is an $n \in \mathbb{N}$ with $f_c^{\circ nk}(z)\in 
V$. Moreover, there is a univalent map $\psi \colon V \to \mathbb{C}$ with $\psi(f_c^{\circ k}(z)) = \overline{\psi(z)}+1/2$, and $\psi(V)$ contains a right half plane. This map $\psi$ is unique up to horizontal translation. 
\end{lemma}
\begin{proof} 
See \cite[Lemma 2.3]{HS}.
\end{proof}

The map $\psi$ will be called an  \emph{antiholomorphic Fatou coordinate} for the petal $V$. The antiholomorphic iterate interchanges both ends of the Ecalle cylinder, so it must fix one horizontal line around this cylinder (the \emph{equator}). The change of coordinate has been so chosen that the equator is the projection of the real axis.  We will call the vertical Fatou coordinate the \emph{Ecalle height}. Its origin is the equator. Of course, the same can be done in the repelling petal as well. The existence of this distinguished real line, or equivalently an intrinsic meaning to Ecalle height, is specific to antiholomorphic maps. 

The Ecalle height of the critical value plays a special role in antiholomorphic dynamics. The next theorem proves the existence of real-analytic arcs of non-cusp parabolic parameters on the boundaries of odd period hyperbolic components of the multicorns.

\begin{theorem}[Parabolic arcs]\label{parabolic arcs}
Let $c_0$ be a parameter such that $f_{c_0}$ has a parabolic orbit of odd period, and suppose that $c_0$ is not a cusp. Then $c_0$ is on a parabolic arc in the  following sense: there  exists a real-analytic arc of non-cusp parabolic parameters $c(t)$ (for $t\in\mathbb{R}$) with quasiconformally equivalent but conformally distinct dynamics of which $c_0$ is an interior point, and the Ecalle height of the critical value of $f_{c(t)}$ is $t$. 
\end{theorem}
\begin{proof} 
See \cite[Theorem 3.2]{MNS}.
\end{proof}

Following \cite{MNS}, we classify parabolic arcs into two types.

\begin{definition}[Root Arcs and Co-Root Arcs]\label{DefRootArc} 
We call a parabolic arc a \emph{root arc} if, in the dynamics of any 
parameter on this arc, the parabolic orbit disconnects the Julia set. 
Otherwise, we call it a \emph{co-root arc}.
\end{definition}

The structure of the hyperbolic components of odd period plays an important role in the global topology of the parameter spaces. Let $H$ be a hyperbolic component of odd period $k\neq1$ (with center $\tilde{c}$) of the multicorn $\mathcal{M}_d^{\ast}$. The first return map of the closure of the characteristic Fatou component of $\tilde{c}$ fixes exactly $d+1$ points on its boundary. Only one of these fixed points disconnects the Julia set, and is the landing point of two distinct dynamical rays at $2k$-periodic (under multiplication by $t\mapsto -dt$ (mod $1$)) angles. Let the set of the angles of these two rays be $S' = \{\alpha_1,\alpha_2 \}$. Each of the remaining $d$ fixed points is the landing point of precisely one dynamical ray at a $k$-periodic (under multiplication by $t\mapsto -dt$ (mod $1$)) angle; let the collection of the angles of these rays be $S = \{ \theta_1, \theta_2, \cdots, \theta_d \}$. We can, possibly after renumbering, assume that $0 < \alpha_1 < \theta_1 < \theta_2 < \cdots < \theta_d < \alpha_2$, and $\alpha_2 - \alpha_1 < \frac{1}{d}$. 

By \cite[Theorem 1.2]{MNS}, $\partial H$ is a simple closed curve consisting of $d+1$ parabolic arcs, and the same number of cusp points such that every arc has two cusp points at its ends. Exactly $1$ of these $d+1$ parabolic arcs is a root arc, and the parameter rays at angles $\alpha_1$ and $\alpha_2$ accumulate on this arc. The rest of the $d$ parabolic arcs are co-root arcs, and each of them contains the accumulation set of exactly one $\mathcal{R}_{\theta_i}^d$. Furthermore, the rational lamination remains constant throughout the closure of the hyperbolic component $H$ except at the cusp points.

The main technical tool used in the proof of the non-trivial accumulation of parameter rays is the perturbation of antiholomorphic parabolic points. We now briefly recall the concepts of near-parabolic antiholomorphic Fatou coordinates and the transit map. Our discussion will roughly follow \cite[\S 4]{HS}. The technique of perturbation of antiholomorphic parabolic points will allow us to transfer information from the dynamical planes to the parameter plane. A more general account on the theory of perturbation of parabolic points can be found in \cite{D2,Shi}.

For the rest of this section, we will assume that $c$ is a non-cusp parabolic parameter of odd period $k$ lying on a parabolic arc $\mathcal{C}$ on the boundary of a hyperbolic component $H$ (of period $k$). All of our discussions will concern the characteristic parabolic point of $f_c$. We will denote an attracting petal of $f_c$ by $V_c^{\mathrm{in}}$, and a repelling petal of $f_c$ by $V_c^{\mathrm{out}}$. There exists an open neighborhood $U$ of $c$ (in the parameter plane) such that for all $c' \in U^{-}:= U \setminus \overline{H}$, the characteristic parabolic point splits into two simple periodic points, and the perturbed Fatou coordinates can be followed throughout $U^-$. More precisely, for $c^{\prime} \in U^{-}$, there exist an incoming domain $V^\mathrm{in}_{c'}$, and an outgoing domain $V^\mathrm{out}_{c'}$ (such that they are disjoint) having the two simple periodic points on their boundaries. There exists a curve joining the two simple periodic points, which we call the ``gate'', such that the points in the incoming domain eventually transit through the gate, and escape to the outgoing domain (as shown in Figure \ref{gate}). Furthermore, there exist injective holomorphic maps $\psi_{c'}^{\mathrm{in}}: V^\mathrm{in}_{c'} \rightarrow \mathbb{C}$, and $\psi_{c'}^{\mathrm{out}}: V^\mathrm{out}_{c'} \rightarrow \mathbb{C}$ such that $\psi_{c'}^{\mathrm{in}}(f_{c'}^{\circ 2k}(z))=  \psi_{c'}^{\mathrm{in}}(z) +1$, whenever $z$ and $f_{c'}^{\circ 2k}(z)$ are both in $V^\mathrm{in}_{c'}$ (and the same is true for $\psi_{c'}^{\mathrm{out}}$). Since the map $f_{c'}^{\circ k}$ commutes with $f_{c'}^{\circ 2k}$, it induces antiholomorphic self-maps from $C^\textrm{in}_{c'}:= V^\textrm{in}_{c'}/f^{\circ 2k}_{c'}$ (respectively $C^\textrm{out}_{c'}:= V^\textrm{out}_{c'}/f^{\circ 2k}_{c'}$) to itself. As $f_{c'}^{\circ k}$ interchanges the two periodic points at the ends of the gate, it interchanges the ends of the cylinders, so it must fix a (necessarily unique) closed geodesic in the cylinders $\mathbb{C}/\mathbb{Z}$. This is similar to the situation at the parabolic parameter $c$, so we will call this invariant geodesic the equator. As for the parabolic parameter $c$, we will choose our Fatou coordinates such that they map the equators to the real line. Thus we can again define Ecalle height as the imaginary part in these Fatou coordinates. We will denote the Ecalle height of a point $z \in C^\textrm{in/out}_{c'}$ by $E(z)$.

We need to normalize our Fatou coordinates (recall that after we decide that the Fatou coordinate maps the equator to the real line, we are left with one real additive degree of freedom of the Fatou coordinate). We can choose two continuous functions $a, b:\overline{U^{-}} \rightarrow \mathbb{C}$ such that $a(c')$ (respectively $b(c')$) lies on the incoming (respectively outgoing) equator in $V^\mathrm{in}_{c'}$ (respectively in $V^\mathrm{out}_{c'}$), for all $c' \in \overline{U^{-}}$. We can normalize $\psi_{c'}^{\mathrm{in}}$ and $\psi_{c'}^{\mathrm{out}}$ by the requirements $\psi_{c'}^{\mathrm{in}}(a(c'))=0$ and  $\psi_{c'}^{\mathrm{out}}(b(c'))=0$, and assume that $\psi_{c'}^{\mathrm{out}}(V^\mathrm{out}_{c'})$ contains the vertical bi-infinite strip $\left[0,1\right]\times \mathbb{R}$. . With these normalization, we have that $(c',z) \mapsto \psi_{c'}^{\mathrm{in}}(z)$ and $(c',z)\mapsto \psi_{c'}^{\mathrm{out}}(z)$ are continuous functions on the open sets $\mathcal{V}^{\mathrm{in}}:=\{ (c',z): z \in V^\mathrm{in}_{c'}\}$ and  $\mathcal{V}^{\mathrm{out}}:=\{ (c',z): z \in V^\mathrm{out}_{c'}\}$ (respectively) in $\overline{U^{-}} \times \mathbb{C}$ (for $c'\in \partial H \cap \overline{U^{-}}$, $\psi_{c'}^{\mathrm{in}}$ and $\psi_{c'}^{\mathrm{out}}$ are respectively the attracting and repelling Fatou coordinates for $f_{c'}$ as in Lemma \ref{normalization of fatou}) \cite[\S 17]{D2}.

It follows that for every $c' \in U^{-}$, the quotients $C^\textrm{in}_{c'} := V^\textrm{in}_{c'}/f^{\circ 2k}_{c'}$ and $C^\textrm{out}_{c'} := V^\textrm{out}_{c'}/f^{\circ 2k}_{c'}$ (the quotients of $V^\textrm{in}_{c'}$ and $V^\textrm{out}_{c'}$ by the dynamics, identifying points that are on the same finite orbits entirely in $V^\textrm{in}_{c'}$ or in $V^\textrm{out}_{c'}$) are complex cylinders isomorphic to $\mathbb{C}/\mathbb{Z}$. The isomorphisms are given by Fatou coordinates which depend continuously on the parameter throughout $U^{-}$.

\begin{figure}[ht!]
\centering
\includegraphics[scale=0.36]{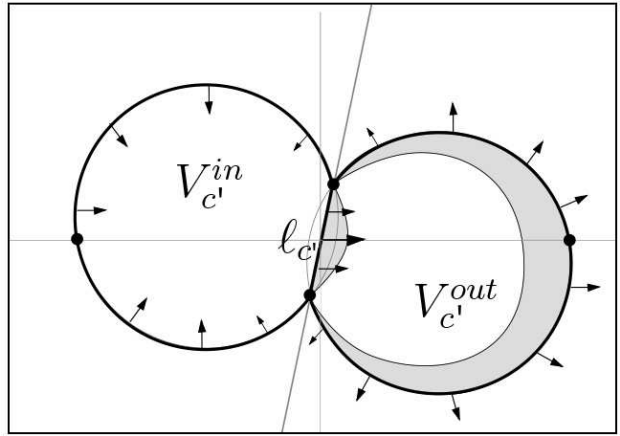}
\caption{The typical dynamical picture after perturbation of the parabolic point (Figure courtesy Dierk Schleicher).}
\label{gate}
\end{figure}

For $c'\in U^{-}$, the incoming and the outgoing cylinders are isomorphic to each other by a natural biholomorphism, namely $f^{\circ 2k}_{c'}$. This isomorphism is called the ``transit map'', and is denoted by $T_{c'}$. The transit map clearly depends continuously on the parameter $c' \in U^{-}$. It maps the fixed geodesic of the incoming cylinder to the fixed geodesic of the outgoing cylinder, and preserves the upper (respectively lower) ends of the cylinders. Thus it must preserve Ecalle heights. The existence of this special isomorphism allows us to relate the Ecalle heights of points in the incoming and the outgoing cylinders, and is one of the principal tools in our study.

Finally, we are in a position to state the key technical lemma that helps us to transfer dynamical information at a parabolic parameter to the parameter plane. One can define the disjoint unions :

\begin{align*}
\displaystyle C^\textrm{in} = \bigsqcup_{c' \in \overline{U^{-}}} C_{c'}^\textrm{in},\ \mathrm{and}\  C^\textrm{out} =\displaystyle \bigsqcup_{c' \in \overline{U^{-}}} C_{c'}^\textrm{out}.
\end{align*}

We topologize $C^\textrm{in}$ and $C^\textrm{out}$ by requiring the following trivializing maps $\Theta^{\mathrm{in}}$ and $\Theta^{\mathrm{out}}$ to be homeomorphisms:

\begin{minipage}{0.4\linewidth}
\begin{align*}
\Theta^{\mathrm{in}} : C^{\textrm{in}} \rightarrow \overline{U^{-}} \times \mathbb{C}/\mathbb{Z}\\
C_{c'}^{\mathrm{in}} \ni z \mapsto (c', \psi_{c'}^{\mathrm{in}}(z))
\end{align*}
\end{minipage}
\begin{minipage}{0.4\linewidth}
\begin{align*}
\Theta^{\mathrm{out}} : C^{\textrm{out}} \rightarrow \overline{U^{-}} \times \mathbb{C}/\mathbb{Z}\\
C_{c'}^{\mathrm{out}} \ni z \mapsto (c', \psi_{c'}^{\mathrm{out}}(z))
\end{align*}
\end{minipage}
\vspace{2mm}

Thus, $C^\textrm{in}$ and $C^\textrm{out}$ are topologically trivial bundles over $\overline{U^{-}}$ with fibers isomorphic to $\mathbb{C}/\mathbb{Z}$. Moreover, the representative of $a(c')$ (respectively $b(c')$) in $C_{c'}^{\mathrm{in}}$ (respectively in $C_{c'}^{\mathrm{out}}$) corresponds to the origin in $\mathbb{C}/\mathbb{Z}$. Note that $\mathbb{C}/\mathbb{Z}$ is equipped with a dynamically marked circle $\mathbb{R}/\mathbb{Z}$, which corresponds to the image of the equator under the Fatou coordinate. Choose a smooth curve $s \mapsto c(s)$ in $\overline{U}$ (in parameter space), parametrized by $s\in [0,\delta]$ for some $\delta >0$, with $c(0) = c$ and $c(s) \in U^{-}$ for $s > 0$. Choose a smooth curve $s \mapsto \zeta(s)$ (in the dynamical planes, typically the critical value), also defined for $s\in [0,\delta]$ such that $\zeta(s) \in V^\textrm{in}_{c(s)}$ for all $s\in [0,\delta]$. Then $s \mapsto \zeta(s)$ induces a map $\sigma : [0,\delta] \rightarrow C^\textrm{in}$ with $\sigma(s)\in C^\textrm{in}_{c(s)}$. The following was proved in \cite[Proposition 4.8]{HS}.

\begin{lemma}[Limit of Perturbed Fatou Coordinates]\label{spiral}
The curve $\gamma := s \mapsto T_{c(s)}(\sigma(s))$
in $C^\textrm{out}$, parametrized by $s \in (0, \delta]$, spirals as $s \downarrow 0$ towards the circle in $C_c^\textrm{out}$ at Ecalle height $E(\sigma(0))$.
\end{lemma}  

Before giving a formal proof of the lemma, let us explain its intuitive meaning. We define the \emph{phase} to be a continuous lift $\tilde{\phi}:(0,\delta] \to \mathbb{R}$ of
\begin{align*}
\phi : \left(0,\delta\right] \rightarrow \mathbb{R}/\mathbb{Z}\\
s \mapsto \Re(\pi_2(\Theta^{\mathrm{out}}(\gamma(s)))).
\end{align*}
As $s$ tends to $0$, i.e.\ as we march closer to the parabolic parameter $c$, $\zeta(s) \in V^\textrm{in}_{c(s)}$ in the dynamical plane of $c(s)$ takes larger and larger number of iterates to pass through the gate. In other words, the number of fundamental domains that $\zeta(s)$ has to cross in order to escape through the gate tends to $+\infty$, as $s$ tends to $0$. This implies that the \emph{phase} $\tilde{\phi}(s)$ goes to $-\infty$ as $s$ tends to $0$, and the image of $\gamma$ accumulates exactly on the circle on $C^{\textrm{out}}_c$
at Ecalle height $E(\sigma(0))$.

\begin{proof}
We first need to choose a continuous lift $\tilde{\phi} : \left(0,\delta\right]\rightarrow \mathbb{R}$ of $\phi$. For $s \in \left(0,\delta\right]$, let $N_s$ be the least integer such that $\alpha(s):= f_{c(s)}^{\circ 2kN_s}(\zeta(s))$ lies in $V_{c(s)}^{\mathrm{out}}$, and satisfies $\Re(\psi_{c(s)}^{\mathrm{out}}(\alpha(s)))\geq 0$. For any fixed $n$, $f_{c(s)}^{\circ 2kn}(\zeta(s)) \to f_c^{\circ 2kn}(\zeta(0))$ as $s\to 0$. But in the dynamical plane of $f_c$, the forward orbit of $\zeta(0)$ under $f_c^{\circ 2k}$ forever lives in the attracting petal $V_{c(0)}^{\mathrm{in}}$. Since $\alpha(s)\notin V^\mathrm{in}_{c(s)}$, this implies that $N_s \to +\infty$ as $s\to 0$.

We can now choose $\tilde{\phi}$ to be 
\begin{align*}
\left(0,\delta\right] \ni s \mapsto \Re(\psi_{c(s)}^{\mathrm{out}}(\alpha(s)))-N_s.
\end{align*}
In fact, $s\mapsto \Re(\psi_{c(s)}^{\mathrm{out}}(\alpha(s)))$ has infinitely many jump discontinuities, but subtracting $N_s$ makes it continuous. 

By our normalization of Fatou coordinates, $\Re(\psi_{c(s)}^{\mathrm{out}}(\alpha(s))) \in \left[0,1\right)$. It now follows that $\tilde{\phi}(s) \to -\infty$ as $s\to 0$.

The fact that the image of $\gamma$ accumulates exactly on the circle on $C^{\textrm{out}}_c$ at Ecalle height $E(\sigma(0))$ follows from the continuity of the Fatou coordinates (in particular, from the continuity of Ecalle height) and the fact that the transit map preserves Ecalle heights.
\end{proof}

We will conclude this section with an analysis of some special properties of horn maps in the antiholomorphic setting. Once again, we exploit the symmetry between the upper and lower ends of the Ecalle cylinders provided by the antiholomorphic return map. For the sake of completeness, we include the basic definitions and properties of horn maps. More comprehensive accounts on these ideas can be found in \cite[\S 2]{BE}.

The characteristic parabolic point $z_c$ (say) of $f_c$ has exactly two petals, one attracting and one repelling (denoted by $\mathcal{P}_{\textrm{att}}$ and $\mathcal{P}_{\textrm{rep}}$ respectively). The intersection of the two petals has two connected components. We denote by $\mathcal{U}^+$ the connected component of $\mathcal{P}_{\textrm{att}} \cap \mathcal{P}_{\textrm{rep}}$ whose image under the Fatou coordinates is contained in the upper half-plane, and by $\mathcal{U}^-$ the one whose image under the Fatou coordinates is contained in the lower half-plane. We define the ``sepals" $\mathcal{S}^{\pm}$ by
\begin{align*}
\displaystyle \mathcal{S}^{\pm} = \bigcup_{n \in \mathbb{Z}} f_c^{\circ 2nk} (\mathcal{U}^{\pm})
\end{align*}
Note that each sepal contains a connected component of the intersection of the attracting and the repelling petals, and they are invariant under the first holomorphic return map of the parabolic point. The attracting Fatou coordinate $\psi_{\textrm{att}}$ (respectively the repelling Fatou coordinate $\psi_{\textrm{rep}}$) can be extended to $\mathcal{P}_{\textrm{att}} \cup \mathcal{S}^+ \cup \mathcal{S}^-$ (respectively to $\mathcal{P}_{\textrm{rep}} \cup \mathcal{S}^+ \cup \mathcal{S}^-$) such that they conjugate the first holomorphic return map to the translation $\zeta \mapsto \zeta+1$.

\begin{definition}[Lifted horn maps]
Let us define $V^- = \psi_{\textrm{rep}}(\mathcal{S}^-)$, $V^+ = \psi_{\textrm{rep}}(\mathcal{S}^+)$, $W^- = \psi_{\textrm{att}}(\mathcal{S}^-)$ and $W^+ =\psi_{\textrm{att}}(\mathcal{S}^+)$. Then, denote by $H^-_c : V^- \rightarrow W^-$ the restriction of $ \psi_{\textrm{att}} \circ \psi_{\textrm{rep}}^{-1}$ to $V^-$ and by $H^+_c : V^+ \rightarrow W^+$ the restriction of $\psi_{\textrm{att}} \circ \psi_{\textrm{rep}}^{-1}$ to $V^+$. We refer to $H^{\pm}_c$ as lifted horn maps for $f_c$ at $z_c$.
\end{definition}

The regions $V^{\pm}$ and $W^{\pm}$ are invariant under translation by $1$. Moreover, the asymptotic development of the Fatou coordinates implies that the regions $V^+$ and $W^+$ contain an upper half-plane, whereas the regions $V^-$ and $W^-$ contain a lower half-plane. Consequently, under the projection $\pi : \zeta \mapsto w = \exp(2i\pi \zeta)$, the regions $V^+$ and $W^+$ project to punctured neighborhoods $\mathcal{V}^+$ and $\mathcal{W}^+$ of $0$, whereas $V^-$ and $W^-$ project to punctured neighborhoods $\mathcal{V}^-$ and $\mathcal{W}^-$ of $\infty$. 

The lifted horn maps $H^{\pm}_c$ satisfy $H_c^{\pm}(\zeta + 1) = H_c^{\pm}(\zeta) + 1$ on $V^{\pm}$.  Thus, they project to mappings $h_c^{\pm} : \mathcal{V}^{\pm} \rightarrow \mathcal{W}^{\pm}$ such that the following diagram commutes:

\begin{center}
$\begin{CD}
V^{\pm} @>H_c^{\pm}>> W^{\pm}\\
@VV\pi V @VV\pi V\\
\mathcal{V}^{\pm} @>h_c^{\pm}>> \mathcal{W}^{\pm}
\end{CD}$
\end{center}

\begin{definition}[Horn Maps]
The maps $h_c^{\pm}$ are called horn maps for $f_c$ at $z_c$.
\end{definition}

It is well-known that $\exists$ $\eta_c, \eta_c' \in \mathbb{C}$ such that $H^+_c(\zeta) \approx \zeta + \eta_c$ when $\Im(\zeta) \rightarrow +\infty$, and $H^-_c(\zeta) \approx \zeta + \eta_c'$ when $\Im(\zeta) \rightarrow -\infty$. This proves that $h_c^+ (w) \rightarrow 0$ as $w \rightarrow 0$. Thus, the horn map $h_c^+$ extends analytically to $0$ by $h^+_c (0) = 0$. One can show similarly that the horn map $h_c^-$ extends analytically to $\infty$ by $h_c^- (\infty) = \infty$. Observe that the constants $\eta_c$ and $\eta_{c'}$ are, in general, not well-defined as they depend on particular normalizations of the Fatou coordinates. However, in the antiholomorphic situation, we can and will choose the normalizations of Fatou coordinates described in Lemma \ref{normalization of fatou}, and these Fatou coordinates conjugate the first (antiholomorphic) return map in both petals to $\zeta \mapsto \overline{\zeta}+1/2$. This choice involves an adjustment of the vertical degree of freedom of the Fatou coordinates. Consequently, the two lifted horn maps $H^+_c$ and $H^-_c$ are conjugated to each other by $\zeta \mapsto \overline{\zeta}+\frac{1}{2}$. It follows that $\eta_c' = \overline{\eta_c}$.  Note that with the chosen normalizations of the Fatou coordinates, the imaginary parts of $\eta_c$ and $\eta_c'$ (which are the asymptotic vertical translation constants of the lifted horn maps) become well-defined real numbers. We will study the asymptotic behavior of $\eta_c$ as $c$ tends to the ends of the parabolic arc $\mathcal{C}$.

\begin{lemma}\label{asymptotics of horn}
$\Im(\eta_c) \rightarrow +\infty$ as $c$ tends to the ends of the parabolic arc $\mathcal{C}$.
\end{lemma}
\begin{proof}
It follows from the symmetry of the two lifted horn maps that the two horn maps $h^+_c$ and $h^-_c$ which are defined respectively in neighborhoods of $0$ and of $\infty$ are conjugated by $w \mapsto -1/\overline{w}$, and they asymptotically look like $w \mapsto \exp(2 \pi i \eta_c) w$ and  $w \mapsto \exp(2\pi i \overline{\eta_c}) w$ respectively. Clearly, $(h^+_c)'(0) =  \exp(2\pi i \eta_c)$ and $(h^-_c)'(\infty) =  \exp(-2\pi i \overline{\eta_c})$. By \cite[Proposition 1]{BE}, $\displaystyle$ $\exp(-4 \pi \Im(\eta_c)) = (h^+_c)'(0) (h^-_c)'(\infty) = \exp(4\pi^2 (1-\iota_c))$, where $\iota_c$ is the holomorphic fixed-point index of $f_c^{\circ 2k}$ at the parabolic fixed point $z_c$. Towards the ends of a parabolic arc, the fixed-point index $\iota_c$ at the characteristic parabolic point tends to $+\infty$ in $\mathbb{R}$ \cite[Proposition 3.7]{HS}. Hence, $\Im(\eta_c) \rightarrow +\infty$ towards the ends of a parabolic arc. 
\end{proof}
\section{Wiggling of Parameter Rays}\label{wiggle}
The goal of this section is to prove Theorem \ref{most rays wiggle}. We begin with a couple of definitions.

\begin{definition}[Accumulation Set of a Ray]
The accumulation set of a parameter ray $\mathcal{R}_{\theta}^d$ of $\mathcal{M}_d^*$ is defined as $L_{\mathcal{M}_d^*}(\theta) := \overline{\mathcal{R}_{\theta}^d} \bigcap \mathcal{M}_d^*$.
\end{definition}

\begin{definition}[Rational Lamination]
The rational lamination of an antiholomorphic polynomial $f_c$ (with connected Julia set) is defined as an equivalence relation on $\mathbb{Q}/\mathbb{Z}$ such that $\theta_1 \sim \theta_2$ if and only if the dynamical rays $R_c(\theta_1)$ and $R_c(\theta_2)$ land at the same point of $J(f_c)$. It is denoted by $\mathcal{RL}(f_c)$.
\end{definition}

\begin{lemma}\label{fixed angles land}
The parameter rays $\{ \mathcal{R}_{t}^d : t = 0, \frac{1}{d+1}, \frac{2}{d+1}, \cdots, \frac{d}{d+1} \}$ of $\mathcal{M}_d^{\ast}$ land.
\end{lemma}

\begin{proof}
Let, $\omega = \exp(\frac{2\pi i}{d+1})$. The antiholomorphic polynomials $f_c$ and $f_{\omega c}$ are conformally conjugate via the linear map $z \mapsto  \omega z$. It follows that $\mathcal{M}_d^{\ast}$ has a $(d+1)$-fold rotational symmetry, and $f_c \sim f_{\omega c} \sim f_{\omega^2 c} \sim \cdots \sim f_{\omega^d c}$.  Also, $K(f_{\omega^j c}) = \omega^j  K(f_c)$, and $J(f_{\omega^j c}) = \omega^j J(f_c)$. The B\"{o}ttcher maps are related by $\omega^j \phi_c(z) = \phi_{\omega^j c}(\omega^j z)$. This reads, in terms of external rays, as $\omega^j R_c(\theta) = R_{\omega^j c} \left( \theta + \frac{j}{d+1} \right)$.

The map $\Phi : \mathbb{C} \setminus \mathcal{M}^{\ast}_d \rightarrow \mathbb{C} \setminus \overline{\mathbb{D}}$, defined by $c \mapsto \phi_c(c)$ (where $\phi_c$ is the B\"{o}ttcher coordinate near $\infty$) is a real-analytic diffeomorphism. This map defines the parameter rays of the multicorns. It follows that $\Phi(\omega^j c) = \omega^j \Phi(c)$.

Now, 
\begin{align*}
 \mathcal{R}^d_0 & = \{ c \in \mathbb{C} : \Phi(c) = r, r>1 \} \\
 &= \{ c \in \mathbb{C} : \frac{1}{\omega^j}\Phi(\omega^j c) = r, r>1 \} \\
 &= \{ c \in \mathbb{C} : \Phi(\omega^j c) = r \exp\left(\frac{2\pi i j}{d+1}\right), r>1 \} \\
 &= \{ c \in \mathbb{C} : \omega^j c \in \mathcal{R}_{\frac{ j}{d+1}}^d \} \\
 &= \frac{1}{\omega^j} \mathcal{R}_{\frac{ j}{d+1}}^d.
\end{align*}

Hence, $\omega^j \mathcal{R}^d_0 = \mathcal{R}_{\frac{ j}{d+1}}^d$. Since $\mathcal{R}^d_0$ is a subset of $\mathbb{R}$, it lands. It follows that $\mathcal{R}_{\frac{ j}{d+1}}^d$, being the image of $\mathcal{R}^d_0$ under a rotation, must land as well.
\end{proof}

Having taken care of the parameter rays at fixed angles, we now turn our attention to the parameter rays that accumulate on the boundaries of hyperbolic components of odd period \emph{greater than} $1$. The proof of Theorem \ref{most rays wiggle} is carried out in various steps. Let us sketch the key ideas of the proof to stop the readers from getting lost in the technicalities. We will stick to the terminologies of Section \ref{Secbackground}.

Let $t \in S \cup S'$, and $\mathcal{C}$ be the parabolic arc where the parameter ray $\mathcal{R}^d_t$ accumulates. For every parameter on $\mathcal{C}$, the dynamical ray at angle $t$ lands at the characteristic parabolic point through the unique repelling petal. We first show that if for some $c \in \mathcal{C}$, the dynamical ray $R_c(t)$ projects to a horizontal line under the repelling Fatou coordinate, then the rational lamination of $f_c$ must be invariant under a certain affine transformation. This is achieved by considering a pair of dynamically meaningful involutions in the repelling petal. A simple combinatorial exercise then shows that such invariant laminations can never exist when the period of $H$ is greater than $1$. This proves that for every parameter on $\mathcal{C}$, the projection of the dynamical ray $R_c(t)$ under the repelling Fatou coordinate must traverse a non-degenerate interval of Ecalle heights. The final part of the proof involves a careful parabolic perturbation argument which allows us to transfer the variation of Ecalle heights of the dynamical rays at angle $t$ to the wiggling of the corresponding ray in the parameter plane.  

We denote the repelling Fatou coordinate at the characteristic parabolic point of $f_c$ by $\psi_{\textrm{rep}} : \mathcal{P}_{\textrm{rep}} \rightarrow \mathbb{H}_\textrm{Left}$, and the B\"{o}ttcher coordinate by $\phi_{c} : \mathcal{A}_{\infty}(f_{c}) \rightarrow \hat{\mathbb{C}}\setminus \overline{\mathbb{D}}$.

\begin{lemma}[Invariance of Lamination]\label{horizontal implies invariance of lamination}
Let $c \in \mathcal{C}$, where $\mathcal{C}$ is a parabolic arc on the boundary of $H$. Suppose that the dynamical ray $R_c(t)$ projects to a horizontal line under the repelling Fatou coordinate. Then the rational lamination  $\mathcal{RL} \left( f_{c} \right)$ is invariant under the transformation $s \mapsto 2t - s$. 
\end{lemma}

\begin{remark}
The projection of the basin of infinity onto the repelling Ecalle cylinder is a conformal annulus bounded by fractal structures (the Julia set and the decorations thereof) from above and below, and the projection of the dynamical ray is the unique simple closed geodesic (the core curve) of this conformal annulus. The core curve is a round circle if and only if the annulus is symmetric with respect to this round circle, and heuristically speaking, this is an extremely unlikely situation for a polynomial. This lemma essentially tells that such a miracle could happen only if the polynomial had some strong global symmetry.
\end{remark}

\begin{figure}[!ht]
\label{Fatou and Bottcher}
\begin{minipage}{0.48\linewidth}
\begin{center}
\includegraphics[scale=0.24]{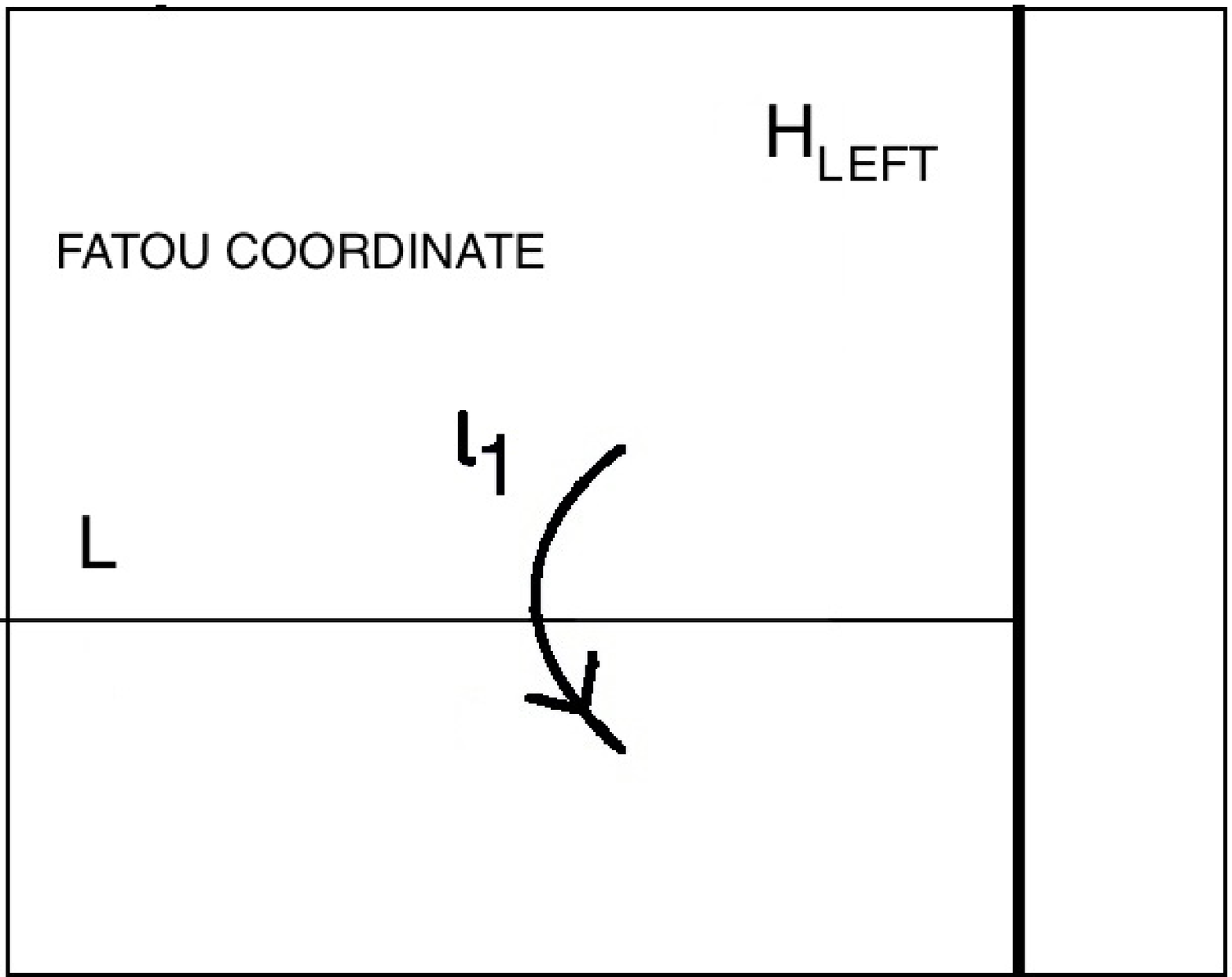}
\end{center}
\end{minipage}
\hspace{2mm}
\begin{minipage}{0.48\linewidth}
\begin{center}
\includegraphics[scale=0.26]{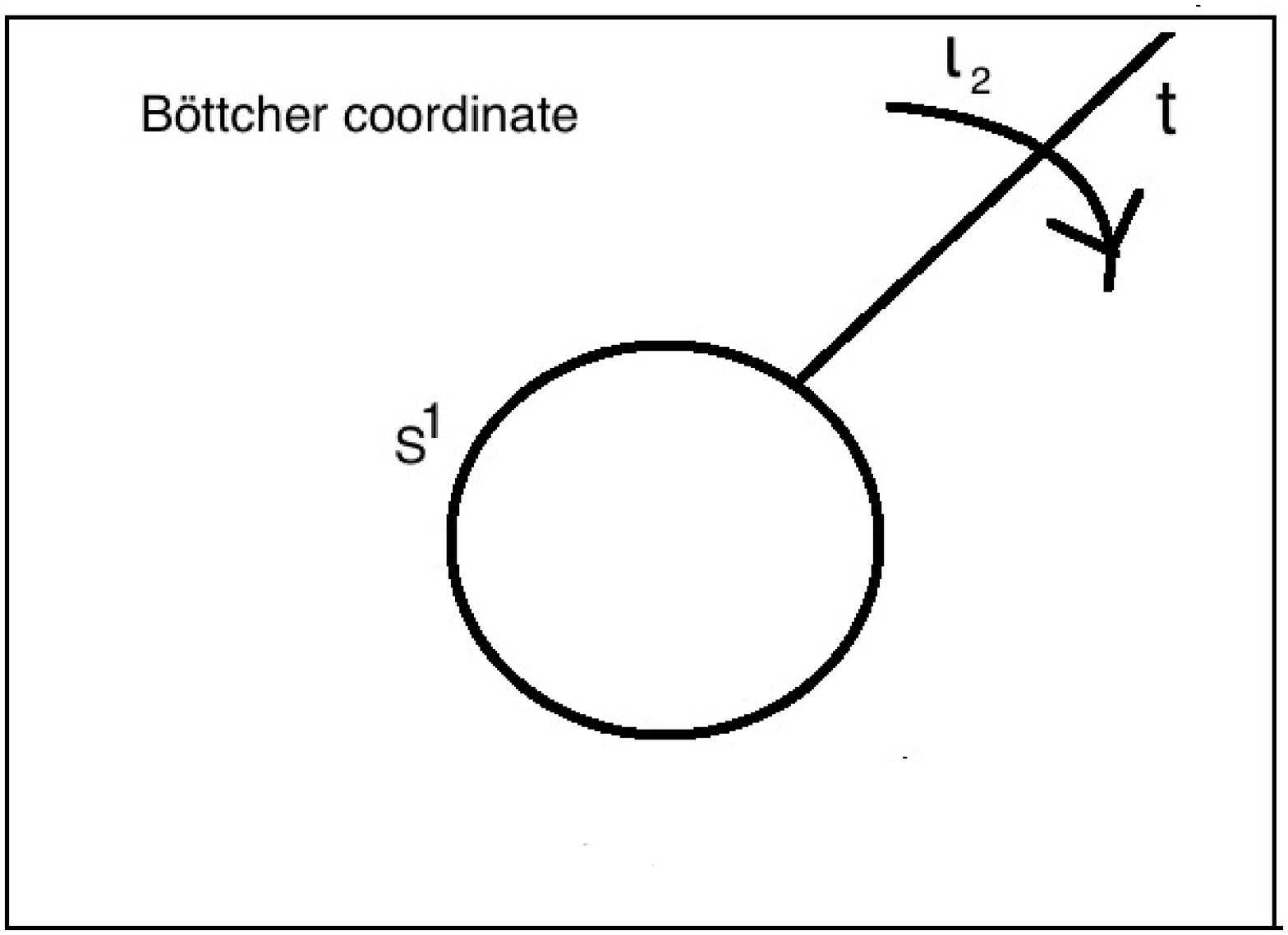}
\end{center}
\end{minipage}

\begin{minipage}{0.22\linewidth}
\begin{flushright}
\includegraphics[scale=0.15]{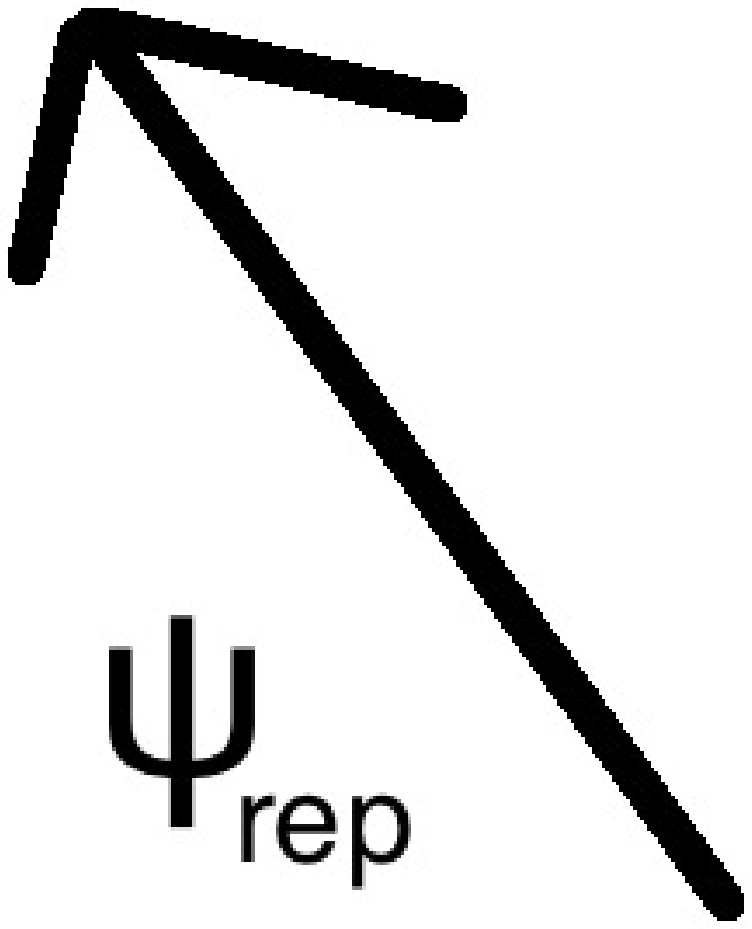}
\end{flushright}
\end{minipage}
\begin{minipage}{0.5\linewidth}
\begin{center}
\includegraphics[scale=0.25]{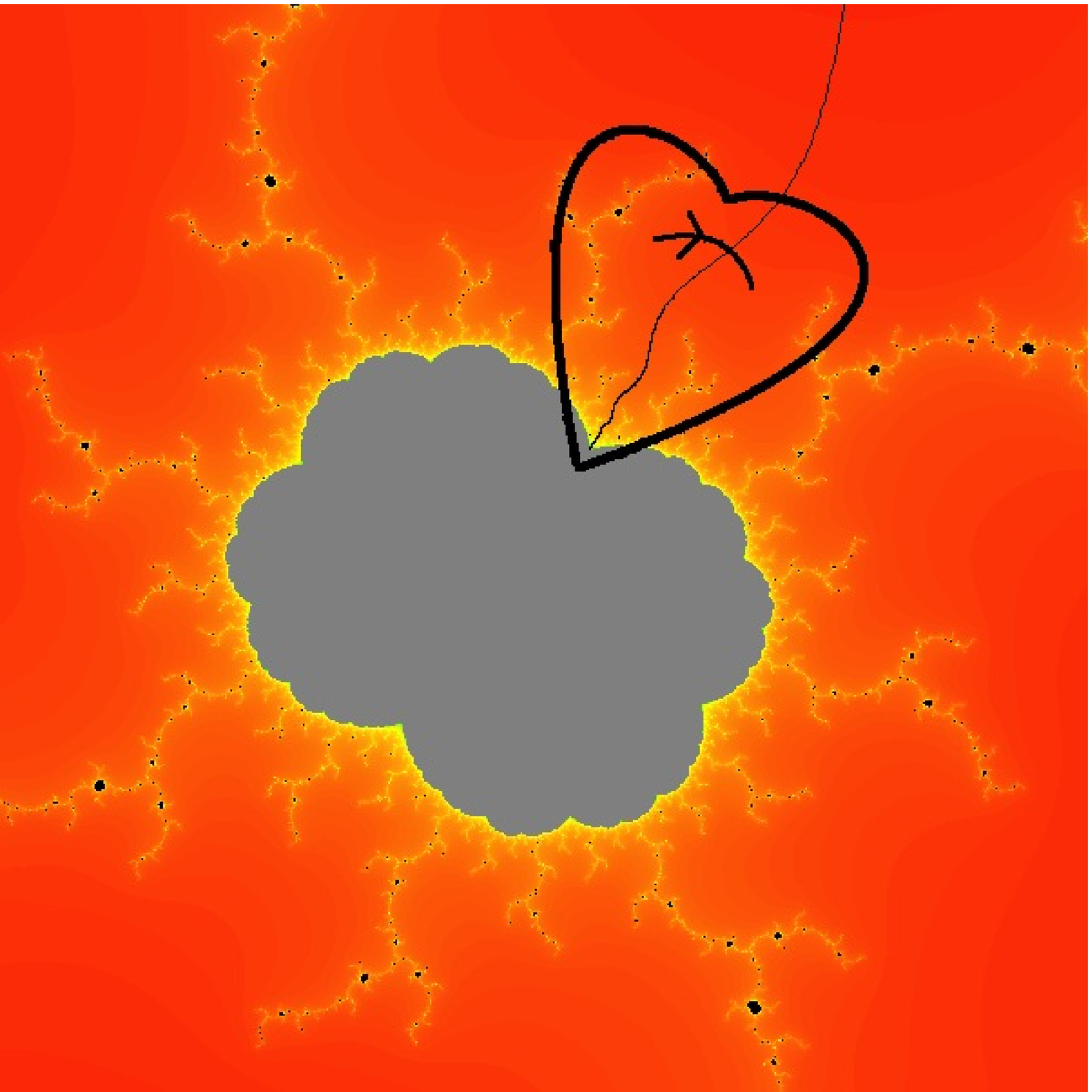}
\end{center}
\end{minipage}
\begin{minipage}{0.22\linewidth}
\includegraphics[scale=0.15]{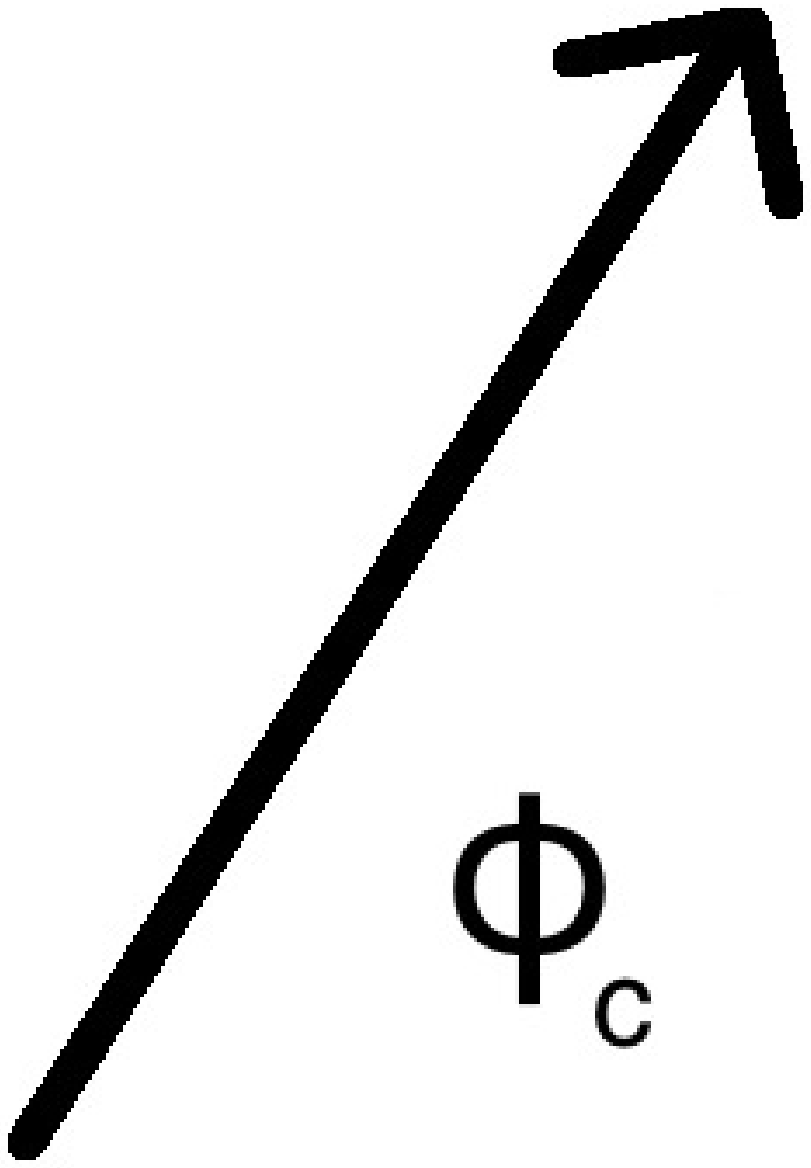}
\end{minipage}
\caption{Top: The reflections with respect to the straight line $L$ (respectively the radial line at angle $t$) defined in the left-half plane (respectively in the exterior of the closed unit disk). Bottom: These reflections transported to the dynamical plane via the Fatou (respectively the B\"{o}ttcher) coordinates agree on their common domain of definitions, namely on $\mathcal{P}_{\textrm{rep}} \cap \mathcal{A}_{\infty}(f_{c})$.}
\end{figure}

\begin{proof}
Let $L$ be the horizontal line which is the image of the dynamical ray $R_c(t)$ under the repelling Fatou coordinate. We denote the reflection in $\mathbb{H}_\textrm{Left}$ with respect to\ $L$ as $\iota_1$. This gives a local anti­holomorphic diffeomorphism $\left( \psi_{\textrm{rep}} ^{-1} \circ \iota_1 \circ \psi_{\textrm{rep}} \right)$ in the domain of definition of the repelling Fatou coordinate. On the other hand, consider the reflection in the B\"{o}ttcher coordinate with respect to the radial line at angle t (denoted by $\iota_2$). This gives an anti­holomorphic diffeomorphism $\left(\phi_c^{-1} \circ \iota_2 \circ \phi_c\right)$ in the basin of infinity $\mathcal{A}_{\infty}(f_{c})$, preserving $R_c(t)$, and mapping a dynamical ray $R_c(s)$ to $R_c(2t-s)$ (see Figure \ref{Fatou and Bottcher}). We will first show that these two diffeomorphisms agree on $\mathcal{P}_{\textrm{rep}} \cap \mathcal{A}_{\infty}(f_{c})$.

Let $\psi_{\textrm{rep}} \circ \phi_{c}^{-­1}$ be defined on a domain $D$ which we can assume to be symmetric with respect to\ the radial line at angle $t$ (under the map $\iota_2$). Since $\psi_{\textrm{rep}} \circ \phi_{c}^{-­1}$ maps the radial line at angle $t$ to a horizontal line in the left half-plane, the Schwarz reflection principle implies that $\psi_{\textrm{rep}} \circ \phi_c^{-­1} (w) = \iota_1 \circ \psi_{\textrm{rep}} \circ \phi_c^{-­1} \circ \iota_2(w) \hspace{1mm} \forall \hspace{1mm} w \in D$. This implies that  $\phi_c^{-1} \circ \iota_2 \circ \phi_c =  \psi_{\textrm{rep}}^{-1} \circ \iota_1 \circ \psi_{\textrm{rep}}$ on $\mathcal{P}_{\textrm{rep}} \cap \mathcal{A}_{\infty}(f_{c})$. Hence, the local antiholomorphic diffeomorphism $\left( \psi_{\textrm{rep}}^{-1} \circ \iota_1 \circ \psi_{\textrm{rep}} \right)$ in the repelling petal maps a co-landing ray pair to another co-landing ray pair. It follows that for rational angles $s_1$ and $s_2$, if $R_c(t+s_1)$ and $R_c(t+s_2)$ land at the same point close to the characteristic parabolic point, then so do $R_c(t-s_1)$ and $R_c(t-s_2)$. This proves the local invariance of the rational lamination under the map $s \mapsto 2t-s$. 

We now spread this local invariance to the entire rational lamination. Let the rational dynamical rays $R_c(s_1)$ and $R_c(s_2)$ co-land. By the density of iterated pre-images in the Julia set, there exists a co-landing rational ray pair $R_c(s'_1)$ and $R_c(s'_2)$ such that their common landing point lies in $\mathcal{P}_{\textrm{rep}}$, and $(-d)^{2mk} s'_1 = s_1$, $(-d)^{2mk} s'_2 = s_2$, for some $m \in \mathbb{N}$. By the local invariance, $R_c(2t-s'_1)$ and $R_c(2t-s'_2)$ co-land. By continuity, $f_{c}^{\circ 2mk} \left( R_c(2t-s'_1) \right) = R_c(2t-s_1)$, and  $f_{c}^{\circ 2mk} \left( R_c(2t-s'_2) \right) = R_c(2t-s_2)$ co-land as well. This completes the proof of the lemma.
\end{proof}

The proof of the above lemma does not use any fact specific to antiholomorphic dynamics, and hence, the conclusion of the lemma holds for any polynomial of degree $d \geq 2$ with connected Julia set. In general, one does not expect the rational lamination of a polynomial to be invariant under such an affine transformation. This can be easily seen in our case, which is the content of the following: 

\begin{lemma}[No Invariant Lamination]\label{invariant lamination only for period 1}
Let $c \in \mathcal{C}\subset\partial H$. Then the rational lamination $\mathcal{RL} \left( f_{c} \right)$ cannot be invariant under the transformation $s \mapsto 2t - s$. 
\end{lemma}
\begin{proof}
We will assume that the rational lamination $\mathcal{RL} \left( f_{c} \right)$ is invariant under the given transformation, and arrive at a contradiction. 
\begin{flushleft}
\emph{Case 1. $t \in S'$}
\end{flushleft}

Without loss of generality, we assume that $t=\alpha_1$. In the dynamical plane of $c$, $R_c(\alpha_1)$ and $R_c(\alpha_2)$ land at a common point (namely, at the characteristic parabolic point). By the invariance, the dynamical rays at angles $\alpha_1$ and $2\alpha_1-\alpha_2$ must land at a common point as well. This is clearly impossible since exactly two rays land at the characteristic parabolic point. 

\begin{flushleft}
\emph{Case 2. $t \in S$, $2t \neq \alpha_1 + \alpha_2$}
\end{flushleft}

Note that in the dynamical plane of $f_c$, the dynamical rays $R_c(\alpha_1)$ and $R_c(\alpha_2)$ land at a common point. By the invariance, the dynamical rays at angles $2t-\alpha_1$ and $2t-\alpha_2$ must land at a common point as well. By our assumption, the rays $R_{c}(2t-\alpha_1)$ and $R_{c}(2t-\alpha_2)$ lie in different connected components of $\mathbb{C} \setminus \left( \overline{R_{c}(\alpha_1)} \bigcup \overline{R_{c}(\alpha_2)} \right)$. This forces four different rays $R_{c}(\alpha_1)$, $R_{c}(\alpha_2)$, $R_{c}(2t-\alpha_1)$, and $R_{c}(2t-\alpha_2)$ to land at a common point (we have used the fact that $0 < \alpha_1 < t < \alpha_2$, and $\alpha_2 - \alpha_1 < 1/2$), which contradicts \cite[Theorem 2.6]{Sa}.

\begin{figure}[!ht]
\noindent
\begin{center}
\begin{minipage}{0.48\linewidth}
\includegraphics[scale=0.19]{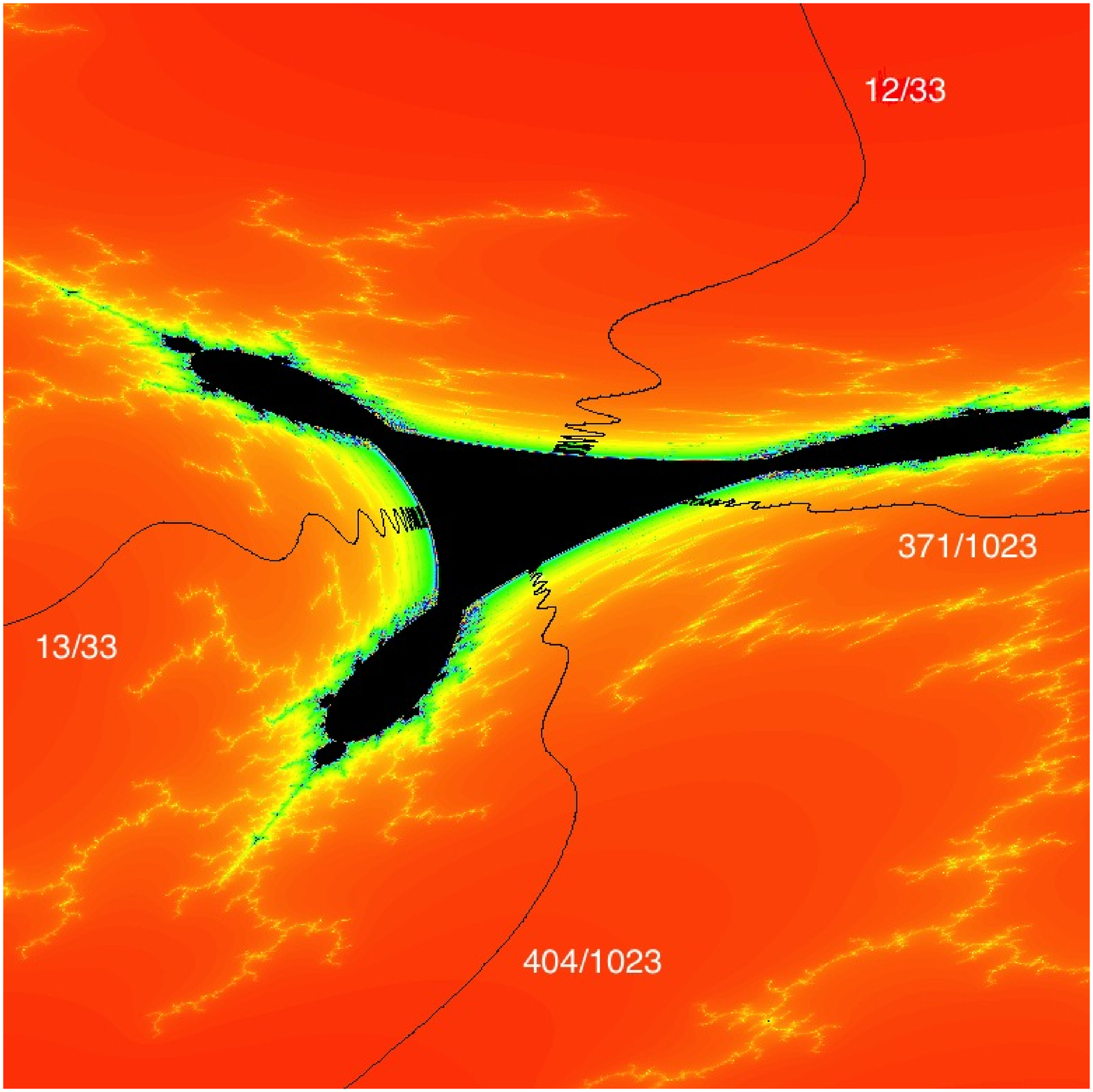}
\end{minipage}
\begin{minipage}{0.48\linewidth}
\includegraphics[scale=0.19]{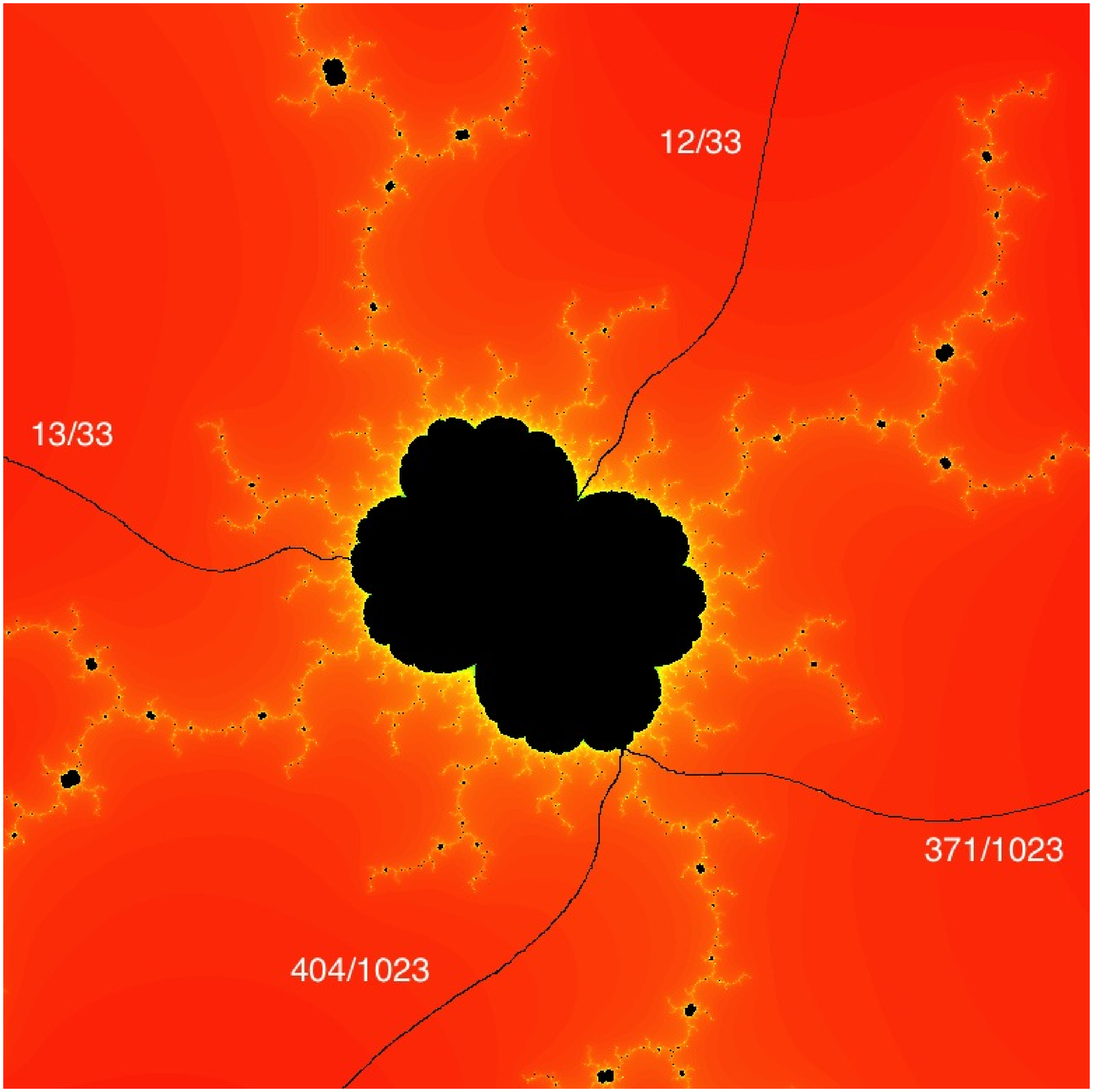}
\end{minipage}
\end{center}
\caption{Left: Parameter rays accumulating on the boundary of a hyperbolic component of period 5 of the tricorn. Right: The corresponding dynamical rays landing on the boundary of the characteristic Fatou component in the dynamical plane of a parameter on the boundary of the same hyperbolic component.}
\label{parameter_dynamical}
\end{figure}

\begin{flushleft}
\emph{Case 3. $t \in S$, $2t = \alpha_1 + \alpha_2$}
\end{flushleft}
We have to work a little harder in this case. Note that for any $i \in \{ 0,1,\cdots,d\}$, the parameter ray $\mathcal{R}^d_{\frac{i}{d+1}}$ lands on the parabolic arc $\mathcal{C}_i$ on the boundary of the period $1$ hyperbolic component. Define the wake $W_i$ to be the connected component of $\mathbb{C} \setminus \{\mathcal{R}_{\frac{i}{d+1}}^d \cup \mathcal{R}_{\frac{i+1}{d+1}}^d \cup \overline{\mathcal{C}_i} \cup \overline{\mathcal{C}_{i+1}}\}$ not containing $0$. Then each parameter in $W_i$ has a repelling periodic orbit admitting the orbit portrait $\mathcal{P}_i = \{ \{\frac{i}{d+1},\frac{i+1}{d+1}\} \}$ such that the dynamical rays at angles $\frac{i}{d+1}$ and $\frac{i+1}{d+1}$ together with their common landing point separate the critical value from the critical point. 

It is easy to see that $H$ must be contained in some $W_j$. In particular, $R_c(\frac{j}{d+1})$ and $R_c(\frac{j+1}{d+1})$ land at the same point (and no other ray lands there), and $\frac{j}{d+1} < t < \frac{j+1}{d+1}$. By the invariance property, the rays $R_c(2t-\frac{j}{d+1})$ and $R_c(2t-\frac{j+1}{d+1})$ must land at the same point as well. By arguing as in Case 2, we can conclude that $t = \frac{2j+1}{2(d+1)}$. Since $t$ is a periodic angle under multiplication by $-d$, it follows that $d$ must be odd. A simple computation now shows that $(-d)^2 t = t$, which contradicts the fact that the period of $t$ is an odd integer $k \neq 1$. This completes the proof of the lemma.
\end{proof}

\begin{remark}
Case 3 of the previous lemma never occurs for the tricorn. For any hyperbolic component $H$ of odd period $k(\neq1)$ of the tricorn, we have $S = \{ \theta_1, \theta_2 \}$, $S' = \{ \alpha_1, \alpha_2\}$, where $\left( 1+2^k \right) \cdot \left( \theta_1 - \alpha_1 \right) = \left( \alpha_2 - \alpha_1 \right) = \left( 1+2^k \right) \cdot \left( \alpha_2 - \theta_2 \right)$ (Compare \cite[Corollary 5.16]{MNS} and Figure \ref{parameter_dynamical}).
\end{remark}

\begin{proof}[Proof of Theorem \ref{most rays wiggle}]
Let $\mathcal{R}^d_t$ accumulates on $\mathcal{C} \subset \partial H$.

\emph{Case 1. $t \in S$.}

Let $c_0$ be the parameter on $\mathcal{C}$ whose critical value has (incoming) Ecalle height $0$. Note that $\mathcal{C}$ must be a co-root arc of $H$. By Lemma \ref{horizontal implies invariance of lamination} and Lemma \ref{invariant lamination only for period 1}, the projection of the dynamical ray $R_{c_0}(t)$ under the repelling Fatou coordinate must traverse a non-degenerate interval of (outgoing) Ecalle heights. Since this dynamical ray is fixed by the first antiholomorphic return map, the interval of (outgoing) Ecalle heights traversed by it must be of the form $\left[-h,h\right]$ for some $h > 0$. Since the rays $R_{c}(t)$ depend uniformly continuously on the parameter $c$, and since the projection into Ecalle cylinders is also continuous, we can choose a small neighborhood $U$ of $c_0$ such that for all $c \in \overline{U} \setminus H$, the projection of the rays $R_{c}(t)$ into the Ecalle cylinders traverse (outgoing) heights at least $\left[-h + \epsilon, h - \epsilon\right]$ (Note that in the outgoing cylinder of $c_0$, $R_{c_0}(t)$ traverses Ecalle heights $\left[-h,h\right]$). To transfer the variation of Ecalle height of $R_{c}(t)$ to wiggling of the parameter ray $\mathcal{R}_{t}^d$, we employ \cite[Proposition 4.8]{HS}. 

Let $c_{h'} \in \mathcal{C}$ be the parameter on $\mathcal{C}$ whose critical value has (incoming) Ecalle height $h'$. We pick a $c_{h'} \in U$ with $h' \in \left[-h + 2 \epsilon , h - 2\epsilon\right]$, and choose any smooth path $\gamma : [0, \delta] \rightarrow \overline{U}$ with $\gamma(0) = c_{h'}$ but so that, except for $\gamma(0)$, the path avoids closures of hyperbolic components of period $k$, and so that the path is transverse to $\mathcal{C}$ at $c_{h'}$. 
 
For $s \in \left[ 0, \delta \right]$, let $z(s)$ be the critical value. For $s > 0$, the critical orbit ``transits'' from the incoming Ecalle cylinder to the outgoing cylinder; as $s \downarrow 0$, the image of the critical orbit in the outgoing Ecalle cylinder has (outgoing) Ecalle height tending to $h' \in \left[ -h + 2 \epsilon , h - 2\epsilon \right] $, while the phase tends to infinity. Therefore, there is $s \in \left(0, \delta_{\epsilon}\right)$ arbitrarily close to $0$ at which the critical value, projected into the incoming cylinder, and sent by the transfer map to the outgoing cylinder, lands on the projection of the ray $R_{\gamma(s)}(t)$. But in the dynamics of $f_{\gamma(s)}$, this means that the critical value is on the dynamical ray $R_{\gamma(s)}(t)$, so $\gamma(s)$ is on the parameter ray $\mathcal{R}_{t}^{d}$. 

Hence, any smooth path starting at $c_{h'} \in \mathcal{C} \cap U$ (with $h' \in \left[-h+ 2 \epsilon , h - 2\epsilon\right]$), and living inside $\overline{U} \setminus \overline{H}$ thereafter, intersects the parameter ray $\mathcal{R}_t^d$ infinitely often. This proves that $\mathcal{R}_t^d$ cannot land.  

\vspace{2mm}

\emph{Case 2. $t \in S'$.}

Let $c_{h} \in \mathcal{C}$ be the parameter on $\mathcal{C}$ whose critical value has (incoming) Ecalle height $h$, and the interval of (outgoing) Ecalle heights traversed by $R_{c_h}(t)$ be $[l_t(c_h)$, $u_t(c_h)]$. By Lemma \ref{horizontal implies invariance of lamination} and Lemma \ref{invariant lamination only for period 1}, $u_t(c_h) > l_t(c_h)$ for every parameter $c_h$. If we knew that there is a parameter $c_h \in \mathcal{C}$ with $h \in \left(l_t(c_h),u_t(c_h)\right)$ (observe that this was automatic in Case 1), then the proof of the wiggling of the parameter ray $\mathcal{R}^d_t$ would proceed exactly as in the previous case. Hence, it suffices to prove the existence of such a parameter $c_h$.

Consider parameters $c_h\in \mathcal{C}$ with $h$ sufficiently close to $+\infty$ such that $c_h$ is a point of a period doubling bifurcation. Perturbing this parameter outside $H$, we again obtain the `open gate' situation, and for such a perturbed parameter, the critical value belongs to a period $2k$ Fatou component which lies above the corresponding dynamical ray at angle $t$. Hence, the critical value exits through the gate staying in the period $2k$ Fatou component all along, and its Ecalle height (here, we do not need to distinguish between incoming and outgoing heights as the height is preserved in the process of transiting through the gate) is necessarily greater than the minimum Ecalle height of the dynamical $t$-ray. Since Fatou coordinates vary continuously under perturbation, this shows that $h \geq l_t(c_h)$ for parameters $c_h$ with $h$ sufficiently close to $+\infty$. Analogously, for parameters $c_h \in \mathcal{C}$ with $h$ sufficiently close to $-\infty$, $h \leq u_t(c_h)$. Once again, these two inequalities together with the fact that the (incoming) Ecalle height of the critical value as well as the interval of (outgoing) Ecalle heights traversed by the dynamical ray at angle $t$ depend continuously on $h$ imply that there is some parameter $c_h$ on $\mathcal{C}$ for which $h \in \left(l_t(c_h),u_t(c_h)\right)$.

This completes the proof of the theorem.
\end{proof}

\section{A Combinatorial Classification}\label{combinatorics_tells_wiggling}
In this section, we will give an algorithm to find whether a rational parameter ray $\mathcal{R}_{t}^d$ lands or oscillates based only on the combinatorics of $t$. The following lemma will be useful for this purpose.

Recall that a finite collection $\mathcal{P} = \{ \mathcal{A}_1 , \mathcal{A}_2, \cdots, \mathcal{A}_k \}$ of subsets of $\mathbb{Q}/\mathbb{Z}$ satisfying the five properties of \cite[Theorem 2.6]{Sa} is called a \emph{formal orbit portrait}.

\begin{lemma}\label{when 2k accumulates}
Let $t \in \mathbb{Q}/\mathbb{Z}$ has period $2k$ under multiplication by $-d$, where $k$ is an odd integer. Consider the collection of finite subsets of $\mathbb{Q}/\mathbb{Z}$ given by $\displaystyle$ $\mathcal{P} = \{ \mathcal{A}_1, \mathcal{A}_2, \cdots, \mathcal{A}_k \}$, where $\mathcal{A}_1 = \{ t, (-d)^k t \}$, and $\mathcal{A}_{i+1} = (-d) \mathcal{A}_i,$ i (mod k). Then the parameter ray $\mathcal{R}_{t}^d$ accumulates on the parabolic root arc of a hyperbolic component of period $k$ if and only if $\mathcal{P}$ satisfies the properties of a formal orbit portrait with characteristic angles $t$ and $(-d)^k t$. 
\end{lemma}
\begin{proof}
If $\mathcal{R}_{t}^d$ accumulates on a sub-arc of the parabolic root arc of an odd period hyperbolic component, then the period of the hyperbolic component must be $k$, and the dynamical rays $R_{\tilde{c}}(t)$ and $R_{\tilde{c}}((-d)^k  t)$ co-land at the dynamical root of the characteristic Fatou component of the center $\tilde{c}$ of $H$. In fact, these are the only rays landing there. It is easy to see that $t$ and $(-d)^k t$ generate the orbit portrait $\mathcal{P}$, and they are also the characteristic angles of the orbit portrait.

The proof of the converse is similar to \cite[Lemma 5.6]{MNS}. For completeness, we work out the details here. Without loss of generality, we can assume that the characteristic arc of $\mathcal{P}$ is $\left( t , (-d)^k t \right)$. Note that any accumulation point of a parameter ray at a $2k$-periodic angle is either a parabolic parameter of odd period $k$ or a parabolic parameter of even period $r$ with $r \vert 2k$ such that the corresponding dynamical ray of period $2k$ lands at the characteristic parabolic point in the dynamical plane of that parameter. Thus the set of accumulation points of $\mathcal{R}_{\theta}^d$, for $\theta \in \mathcal{A}_1 \cup \cdots \cup \mathcal{A}_k$, is contained in $F =($The union of the closures of the finitely many root arcs of period $k) \bigcup ($The finitely many parabolic parameters of even period, and of ray period $2k)$.

Consider the connected components $U_i$ of $\displaystyle \mathbb{C} \setminus \left(\bigcup_{\theta \in \mathcal{A}_1 \cup \cdots \cup \mathcal{A}_k} \mathcal{R}_{\theta}^d \cup F \right)$. There are only finitely many components $U_i$, and they are open. Then for every parameter $c \in U_i$, the co-landing patterns of the dynamical rays $R_c(t)$ with $t \in\mathcal{A}_1 \cup \cdots \cup \mathcal{A}_k$ remain the same \cite[Lemma 2.4]{MNS}.
 
Let $U_1$ be the component which contains all parameters $c$ outside $\mathcal{M}_d^{\ast}$ with external angle $t(c) \in \left( t , (-d)^k t \right)$ (there is such a component as $(t , (-d)^k t)$ does not contain any other angle of $\mathcal{P}$). $U_1$ must have the two parameter rays $\mathcal{R}_{t}^d$ and $\mathcal{R}_{(-d)^k t}^d$ on its boundary. By the proof of \cite[Theorem 3.1]{Sa}, each $c \in U_1 \setminus \mathcal{M}_d^{\ast}$ has a repelling periodic orbit admitting the portrait $\mathcal{P}$. If the two parameter rays at angles $t$ and $(-d)^k t$ do not land at a common point or accumulate on a common root arc, then $U_1$ would contain parameters $c$ outside $\mathcal{M}_d^{\ast}$ with $t(c) \notin \left(t , (-d)^k t \right)$. It follows from  the remark at the end of \cite[\S 3]{Sa} that such a parameter can never admit the orbit portrait $\mathcal{P}$, which contradicts the stability of the co-landing patterns of the dynamical rays $R_c(t)$ (with $t \in\mathcal{A}_1 \cup \cdots \cup \mathcal{A}_k$) throughout $U_i$. Hence, the parameter rays $\mathcal{R}_{t}^d$ and $\mathcal{R}_{(-d)^k t}^d$  must land at a common even period parabolic parameter of ray period $2k$ or accumulate on a common root arc of period $k$ of $\mathcal{M}_d^{\ast}$. But it is easy to see that if $\mathcal{R}_{t}^d$ and $\mathcal{R}_{(-d)^k t}^d$ co-land at a parabolic parameter, then the period of its parabolic orbit must be odd, ruling out the first possibility.
\end{proof} 

\begin{theorem}[Combinatorial Classification] \label{complete picture of rays}
Let $t \in \mathbb{Q}/\mathbb{Z}$.

1) If the period of $t$ under multiplication by $-d$ is $4k$ for some $k \in \mathbb{N}$, then $\mathcal{R}_{t}^d$ lands at a parabolic parameter on the boundary of a hyperbolic component of period $4k$.

2) If the period of $t$ under multiplication by $-d$ is an odd integer $k$, then it lands if $k=1$, and accumulates on a sub-arc (of positive length) of a parabolic co-root arc on the boundary of a hyperbolic component of period $k$ otherwise.

3) If the period of $t$ under multiplication by $-d$ is $2k$ for some odd integer $k$, then it accumulates on a sub-arc (of positive length) of the parabolic root arc of a hyperbolic component of period $k$ if and only if the collection of finite subsets of $\mathbb{Q}/\mathbb{Z}$ given by $\mathcal{P} = \{ \mathcal{A}_1, \mathcal{A}_2, \cdots, \mathcal{A}_k \}$, where $\mathcal{A}_1 = \{ t, (-d)^k t \}$, and $\mathcal{A}_{i+1} = (-d) \mathcal{A}_i,$ i (mod k), satisfies the properties of a formal orbit portrait with characteristic angles $t$ and $(-d)^k t$. Otherwise, it lands at a parabolic parameter on the boundary of a hyperbolic component of period $2k$.

4) If $t$ is strictly pre-periodic under multiplication by $-d$, then $\mathcal{R}_{t}^d$ lands at a Misiurewicz parameter.
\end{theorem}

\begin{proof}
1) See \cite[Lemma 7.2]{MNS}.

2) By \cite[Corollary 5.14]{MNS}, every rational parameter ray at an angle $t$ of odd period $k$ lands/accumulates on a sub-arc of a parabolic co-root arc of period $k$. By Theorem \ref{most rays wiggle}, only the rays at fixed angles land at a single point of a parabolic arc, so the others must accumulate on a sub-arc of positive length.

3) This directly follows from \cite[Lemma 7.2]{MNS}, Lemma \ref{when 2k accumulates}, and Theorem \ref{most rays wiggle}.

4) Arguing as in \cite[Theorem 1.1 (3)]{S1a}, one sees that for any limit point $c$ of $\mathcal{R}_{t}^d$, the critical value $c$ is pre-periodic under $f_c$ with fixed period and pre-period. This implies that all the critical points are strictly pre-periodic (with fixed pre-periods and periods) for the holomorphic polynomial $f_c^{\circ 2}$. Since the accumulation set of a parameter ray is connected, it now suffices to prove that there are only finitely many parameters with these algebraic data.

In fact, it is not hard to see that there are only finitely many pairs of complex numbers $(a, b)$ such that the polynomial $(z^d+a)^d+b$ has strictly pre-periodic critical points with fixed pre-periods and periods. The conditions on the critical points determine a pair of distinct algebraic curves in $\mathbb{C}^2$, and their intersection is contained in the connectedness locus of the family of polynomials of degree $d^2$ (recall that the Julia set of a polynomial is connected if and only if all the critical orbits are bounded). Since the connectedness locus is compact \cite{BH1}, it follows from B\'{e}zout's theorem that the two algebraic curves under consideration must intersect at a finite set of points. This shows that there are only finitely many Misiurewicz parameters with fixed period and pre-period in the space of degree $d$ unicritical antiholomorphic polynomials.
\end{proof}
\section{Undecorated Arcs on The Boundaries of The Multicorns}\label{undecorated}

Recall that every parabolic arc has, at both ends, an interval of positive length at which a bifurcation from a hyperbolic component of odd period $k$ to a hyperbolic component of period $2k$ occurs. The decorations attached to these period $2k$ components accumulate on sub-arcs of positive length of the parabolic arc \cite[Theorem~7.3]{HS}. In this section, we will prove that for the parabolic arcs of period $1$, the accumulation sets of these decorations do not overlap; they stay at a positive distance away from the Ecalle height $0$ parameters.

\begin{proof}[Proof of Theorem \ref{period 1 multicorns}]
Every multicorn $\mathcal{M}^*_d$ contains a parabolic arc $\mathcal{C}$ of period $1$ intersecting the positive real axis at a unique (non-cusp) parameter $c_d=d^{\frac{d}{1-d}} (d-1)$. Note that $f_{c_d}$ has a unique parabolic fixed point $d^{\frac{1}{1-d}}$ on the real line. Due to the rotational symmetries of the multicorns, it suffices to prove the result for this arc. In the dynamical plane of $f_{c_d} = \overline{z}^d + c_d$, the parabolic fixed point has a unique access through the unique repelling petal, and the critical value $c_d$ has incoming Ecalle height $0$ (in fact, the incoming and outgoing equators are both contained in the real line, and so is the critical value). The projection of the Julia set in the repelling cylinder is a pair of disjoint simple closed curves (the Julia set in this case is simply the boundary of the immediate basin of attraction of the parabolic fixed point), and together they bound a cylinder $C$ of finite modulus. This finite modulus cylinder $C$ is the projection of the basin of infinity in the repelling Ecalle cylinder. We will first show that the cylinder $C$ contains a round cylinder containing the equator.

We choose the repelling Fatou coordinate at the parabolic fixed point so that the equator is mapped to the real line. Since $f_{c_d}$ commutes with complex conjugation, our Fatou coordinates also have the same property. This implies that the upper and the lower components (disjoint simple closed curves) of the projection of the Julia set in the repelling Ecalle cylinder are symmetric with respect to the real line. It follows that both these curves stay at a bounded distance away from the real line; in other words, the projection of the basin of infinity in the repelling Ecalle cylinder contains a round cylinder $\mathbb{S}^1 \times \left[ -\epsilon, \epsilon \right]$ for some $\epsilon > 0$. Alternatively, it is easy to see that the dynamical ray $R_{c_d}(0)$ (and its image under the repelling Fatou coordinate) is contained in the real line, and hence coincides with the equator in the repelling petal. This shows that the equator is contained in the basin of infinity. Hence, the projection of the basin of infinity in the repelling Ecalle cylinder contains a horizontal round circle, and thus also contains a round cylinder $\mathbb{S}^1 \times \left[ -\epsilon, \epsilon \right]$ for some $\epsilon > 0$ (compare Figure \ref{round_undecorated}).

\begin{figure}[ht!]
\begin{minipage}{0.48\linewidth}
\begin{center}
\includegraphics[scale=0.15]{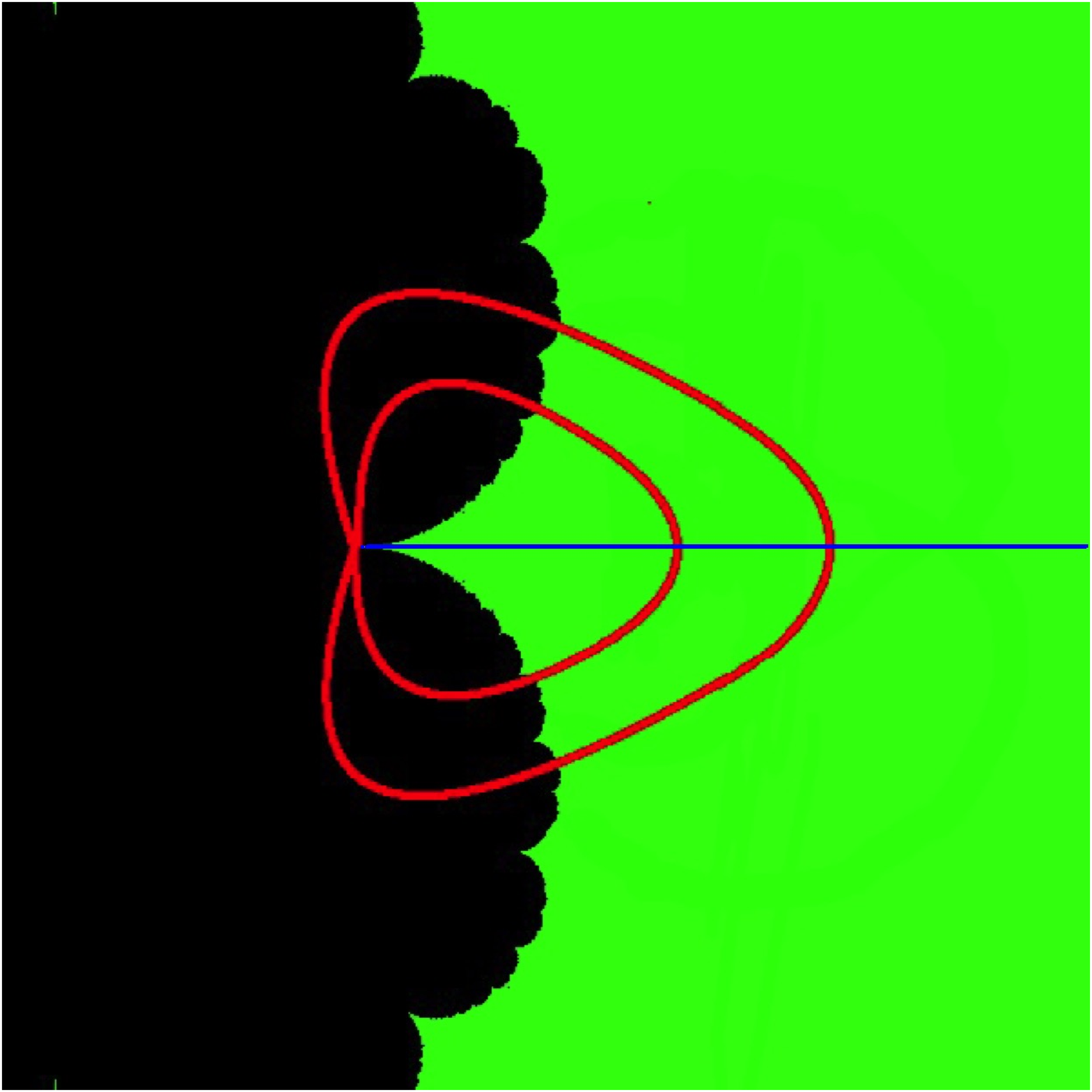}
\end{center}
\end{minipage}
\begin{minipage}{0.48\linewidth}
\begin{flushright}
\includegraphics[scale=0.46]{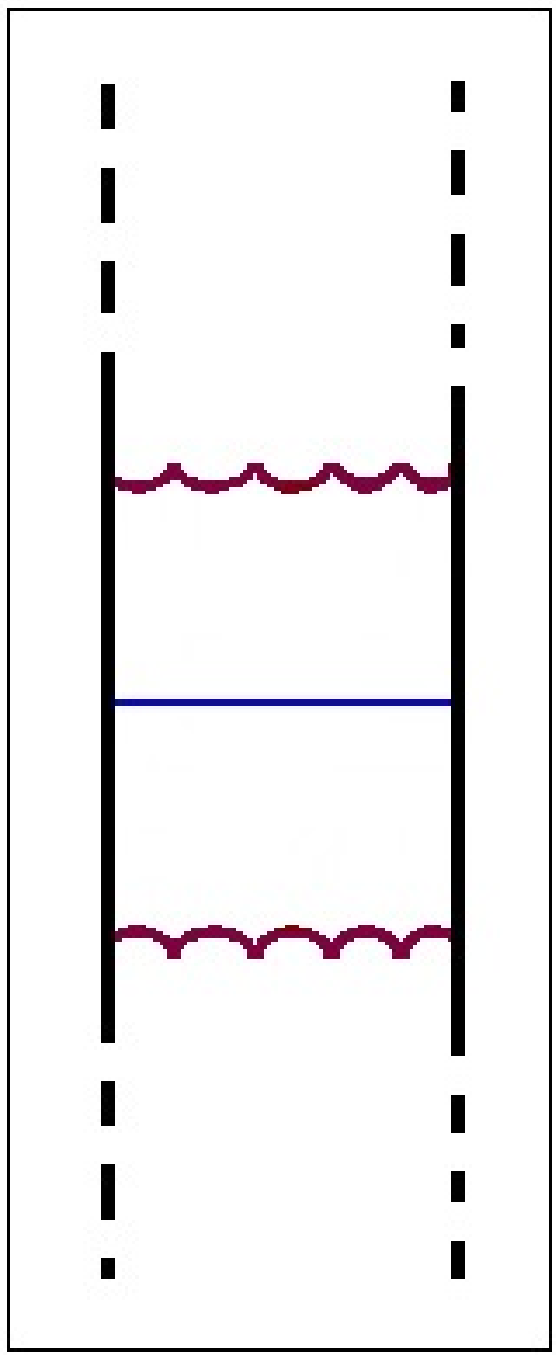}
\end{flushright}
\end{minipage}
\caption{Left: A fundamental domain for the parabolic dynamics of $\overline{z}^2+1/4$ (drawn in red), and the dynamical ray at angle $0$ (drawn in blue) . Right: The corresponding repelling Ecalle cylinder, where the central blue curve is the projection of the dynamical ray at angle $0$. The projection of the basin of infinity to this repelling Ecalle cylinder contains a round annulus.}
\label{round_undecorated}
\end{figure}

The final step is to transfer this round cylinder to an undecorated sub-arc in the parameter plane. It is known that the basin of infinity can not get too small when $c_d$ is perturbed a little bit (compare \cite[Theorem 5.1(a)]{D2}). Since the critical value and the Fatou coordinates depend continuously on the parameter, we can choose a small neighborhood $U$ of $c_d$ such that for all $c \in \overline{U} \setminus H$, the round cylinder $\mathbb{S}^1 \times \left[ -\epsilon/2,  \epsilon/2 \right]$ is contained in projection of the basin of infinity into the repelling Ecalle cylinder (note that in the outgoing cylinder of $c_d$, the round cylinder $\mathbb{S}^1 \times \left[ -\epsilon, \epsilon \right]$ is contained in the projection of the basin of infinity), and such that the critical value of $f_c$ has incoming Ecalle height in $\left[ -\epsilon / 4, \epsilon / 4 \right]$. 

We claim that $\overline{U} \setminus \overline{H}$ is contained in the exterior of the multicorns. Indeed, let $c \in \overline{U} \setminus \overline{H}$. In the dynamical plane of $f_{c}$, the critical orbit of $f_c$ ``transits'' from the incoming Ecalle cylinder to the outgoing cylinder, and the Ecalle height is preserved in the process. By our construction, this would provide with a point of the critical orbit with outgoing Ecalle height in $\left[-\epsilon/4, \epsilon/4 \right]$ in the repelling Ecalle cylinder. But since $c \in \overline{U} \setminus H$, any point in the repelling cylinder with (outgoing) Ecalle height in $\left[-\epsilon/2, \epsilon/2 \right]$ is contained in the projection of the basin of infinity. Therefore, the critical orbit is contained in the basin of infinity; i.e.\ $c \notin \mathcal{M}_d^*$.

This implies that $\overline{U} \cap \mathcal{M}_d^* \subset \overline{H}$. Hence, $\mathcal{C}$ contains a sub-arc containing the Ecalle height $0$ parameter, no point of which is a limit point of further decorations; i.e.\ $\mathcal{C}$ has an undecorated sub-arc. 
\end{proof}

\begin{remark}
a) One can prove the following slightly stronger statement for the tricorn: the parabolic arcs of period $1$ and $3$ contain undecorated sub-arcs. Indeed, in the dynamical plane of a parameter on a parabolic arc of odd period $k$, the projection of the basin of infinity into the repelling Ecalle cylinder is either an annulus of modulus $\frac{\pi}{2k\ln 2}$ or two disjoint annuli, each of modulus $\frac{\pi}{2k\ln 2}$ (depending on whether the parameter is on a co-root or root arc). For $k = 1$ and $3$, this modulus is greater than $1/2$; i.e.\ the corresponding annuli are not too thin. It is well-known (see \cite[Theorem I]{BDH}, for instance) that such a conformal annulus contains a round annulus centered at the origin. In other words, there is an interval $I$ of outgoing Ecalle heights such that in the repelling Ecalle cylinder, the round cylinder $\mathbb{S}^1 \times I$ is contained in the projection of the basin of infinity. One can now prove the existence of undecorated sub-arcs by using the same technique as in Theorem \ref{period 1 multicorns}. 

b) Numerical experiments show that away from the real line, the parabolic arcs of sufficiently high periods of the multicorns do not contain undecorated sub-arcs, rather the accumulation sets of the decorations attached to the two bifurcating hyperbolic components at the ends of such an arc overlap. This overlapping phenomenon would automatically make the corresponding parameter rays wiggle on such arcs. However, we do not know how to prove this statement.
\end{remark}
 
 We finish with a couple of interesting consequences of the previous theorem. 
 
\begin{corollary}\label{centers not dense}
The centers of the hyperbolic components of $\mathcal{M}_d^*$ do not accumulate on the entire boundary of $\mathcal{M}_d^*$. The Misiurewicz parameters are not dense on the boundary of $\mathcal{M}_d^*$.
\end{corollary}

In the last corollary, we show that there are no bifurcations near the Ecalle height $0$ parameters of the parabolic arcs. This was first proved in \cite[Theorem 7.1]{HS}, the present proof is somewhat simpler, and almost readily follows from the previous results. 

\begin{corollary}[No Bifurcation near Ecalle Height Zero]\label{No bifurcation near 0} On every parabolic arc of period $k$, the point with Ecalle height zero has a neighborhood (along the arc) that does not intersect the boundary of a hyperbolic component of period $2k$.
\end{corollary}
\begin{proof}
For the parabolic arcs of period one, the statement readily follows from Theorem \ref{period 1 multicorns}. By the proof of Theorem \ref{most rays wiggle}, the Ecalle height $0$ parameter on any co-root arc has an open neighborhood (along the arc) which lies in the accumulation set of a parameter ray; hence this neighborhood does not intersect the bifurcating period $2k$ components. 

To finish the proof, assume that $\mathcal{C}$ is a root arc. Let $\alpha_1$ and $\alpha_2$ be the angles of the parameter rays accumulating on $\mathcal{C}$. In the dynamical plane of any $c \in \mathcal{C}$, the corresponding dynamical rays land at the characteristic parabolic point through two different accesses in the repelling petal. These two accesses are separated by a parabolic Hubbard tree, which is invariant under the first anti-holomorphic return map. Clearly, the tree either projects to the equator in the repelling cylinder or its projection traverses an interval of Ecalle heights $\left[-a, a\right]$ for some $a>0$. We can, without loss of generality, assume that the dynamical $\alpha_1$ (respectively $\alpha_2$)-ray lies `above' (respectively `below') the hubbard tree (more precisely, this means that the image of the $\alpha_1$-ray under the repelling Fatou coordinate lies in the complementary component of the image of the Hubbard tree containing an upper half plane). We denote the interval of Ecalle heights traversed by $R_c(\alpha_1)$ (respectively, $R_c(\alpha_2)$) by $\left[l_1(c), u_1(c)\right]$ (respectively, $\left[l_2(c), u_2(c)\right]$). It now follows that $u_1(c)>0$, and $l_2(c)<0$ $\forall \hspace{1mm} c \in \mathcal{C}$. Arguing as in case 2 of Theorem \ref{most rays wiggle}, we can find a parameter $c_h \in \mathcal{C}$ (respectively $c_{h'}$) with critical Ecalle height $h>0$ (respectively $h'<0$) so that $h \in \left(l_1(c_h), u_1(c_h)\right)$ (respectively $h' \in \left(l_2(c_{h'}), u_2(c_{h'})\right)$. This implies that $c_h$ and $c_{h'}$ are in the accumulation sets of the parameter rays at angles $\alpha_1$ and $\alpha_2$ respectively. Hence, the convex hull (along $\mathcal{C}$) of the accumulation sets of these two parameter rays contains the Ecalle height $0$ parameter on $\mathcal{C}$. This shows that the accumulation sets of two parameter rays bound the Ecalle height $0$ parameter on every root arc away from the bifurcating period $2k$ components, and completes the proof of the corollary.
\end{proof}

\bibliographystyle{alpha}
\bibliography{w}

\begin{thebibliography}{CHRC89}

\bibitem[BBM15]{BBM}
Araceli Bonifant, Xavier Buff, and John Milnor.
\newblock On antipode preserving cubic maps.
\newblock URL:~\texttt{http://www.math.sunysb.edu/~jack/bbm.pdf}, 2015.

\bibitem[BDH04]{BDH}
Gamaliel Ble, Adrien Douady, and Christian Henriksen.
\newblock Round annuli.
\newblock {\em Contemporary Mathematics: In the Tradition of Ahlfors and Bers,
  {III}}, 355:71--76, 2004.

\bibitem[BE02]{BE}
Xavier Buff and Adam~L. Epstein.
\newblock A parabolic {P}ommerenke-{L}evin-{Y}occoz inequality.
\newblock {\em Fund. Math.}, 172:249--289, 2002.

\bibitem[BH88]{BH1}
Bodil Branner and John~H. Hubbard.
\newblock The iteration of cubic polynomials, part {I}: The global topology of
  parameter space.
\newblock {\em Acta Math.}, 160:143--206, 1988.

\bibitem[CHRC89]{CHRS}
W.~D. Crowe, R.~Hasson, P.~J. Rippon, and P.~E. D.~Strain Clark.
\newblock On the structure of the {M}andelbar set.
\newblock {\em Nonlinearity}, 2, 1989.

\bibitem[Dou94]{D2}
Adrien Douady.
\newblock Does a {J}ulia set depend continuously on the polynomial?
\newblock {\em Complex dynamical systems, Proc. Sympos. Appl. Math.},
  49:91--138, 1994.

\bibitem[GM93]{GM1}
Lisa~R. Goldberg and John Milnor.
\newblock Fixed points of polynomial maps {II}: Fixed point portraits.
\newblock {\em Ann. Scient. \'Ecole Norm. Sup., $ 4^{e}$ s\'erie}, 26:51--98,
  1993.

\bibitem[HS14]{HS}
John Hubbard and Dierk Schleicher.
\newblock {M}ulticorns are not path connected.
\newblock {\em Frontiers in Complex Dynamics: In Celebration of John Milnor's
  80th Birthday}, pages 73--102, 2014.

\bibitem[IM15]{IM}
Hiroyuki Inou and Sabyasachi Mukherjee.
\newblock Discontinuity of straightening in antiholomorphic dynamics.
\newblock manuscript in preparation, 2015.

\bibitem[Ino14]{I1}
Hiroyuki Inou.
\newblock Self-similarity for the tricorn.
\newblock URL:~\texttt{http://arxiv.org/pdf/1411.3081.pdf}, 2014.

\bibitem[KN04]{KN}
Yohei Komori and Shizuo Nakane.
\newblock Landing property of stretching rays for real cubic polynomials.
\newblock {\em Conformal Geometry and Dynamics}, 8:87--114, 2004.

\bibitem[Lav89]{lavaurs_systemes_1989}
Pierre Lavaurs.
\newblock {\em Syst\`{e}mes dynamiques holomorphes: explosion de points
  p\'{e}riodiques paraboliques}.
\newblock PhD thesis, Universit\'{e} de Paris-Sud Centre {d'Orsay}, 1989.

\bibitem[Mil92]{M3}
John Milnor.
\newblock Remarks on iterated cubic maps.
\newblock {\em Experiment. Math.}, 1:5--24, 1992.

\bibitem[Mil00]{M4}
John Milnor.
\newblock On rational maps with two critical points.
\newblock {\em Experiment. Math.}, 9:481--522, 2000.

\bibitem[Mil06]{M1new}
John Milnor.
\newblock {\em Dynamics in one complex variable}.
\newblock Princeton University Press, New Jersey, 3rd edition, 2006.

\bibitem[MNS14]{MNS}
Sabyasachi Mukherjee, Shizuo Nakane, and Dierk Schleicher.
\newblock On {M}ulticorns and {U}nicorns {II}: bifurcations in spaces of
  antiholomorphic polynomials.
\newblock URL:~\texttt{http://arxiv.org/abs/1404.5031}, to appear in `Ergodic
  Theory and Dynamical systems', 2014.

\bibitem[Muk15]{Sa}
Sabyasachi Mukherjee.
\newblock Orbit portraits of unicritical antiholomorphic polynomials.
\newblock {\em Conformal Geometry and Dynamics of the AMS}, 19:35--50, 2015.

\bibitem[Nak93]{Na1}
Shizuo Nakane.
\newblock Connectedness of the {T}ricorn.
\newblock {\em Ergodic Theory and Dynamical Systems}, 13:349--356, 1993.

\bibitem[NS03]{NS}
Shizuo Nakane and Dierk Schleicher.
\newblock On {M}ulticorns and {U}nicorns {I} : Antiholomorphic dynamics,
  hyperbolic components and real cubic polynomials.
\newblock {\em International Journal of Bifurcation and Chaos}, 13:2825--2844,
  2003.

\bibitem[Sch00]{S1a}
Dierk Schleicher.
\newblock Rational parameter rays of the {M}andelbrot set.
\newblock {\em Ast\'erisque}, 261:405--443, 2000.

\bibitem[Shi00]{Shi}
Mitsuhiro Shishikura.
\newblock Bifurcation of parabolic fixed points.
\newblock In {\em The Mandelbrot Set, Theme and Variations}, London
  Mathematical Society Lecture Note Series (No. 274), pages 325--364. Cambridge
  University Press, 2000.

\bibitem[Tan06]{Lei}
Lei Tan.
\newblock Stretching rays and their accumulations, following {P}ia {W}illumsen.
\newblock In {\em Dynamics on the Riemann Sphere: A Bodil Branner Festschrift},
  pages 183--208. European Mathematical Society, 2006.

\end{thebibliography}

\end{document}